\documentclass{amsart}
\usepackage[utf8]{inputenc}
\usepackage{amssymb}
\usepackage{xypic}
\usepackage{enumitem}
\usepackage{amsthm}
\usepackage[hidelinks]{hyperref}
\usepackage[capitalize]{cleveref}

\crefname{enumi}{part}{parts}
\Crefname{enumi}{Part}{Parts}

\newcommand{\CC}{\mathbb C}

\newcommand{\PP}{\mathbb P}
\renewcommand{\SS}{\mathcal{S}}
\newcommand{\OO}{\mathcal{O}}
\newcommand{\QQ}{\mathcal{Q}}
\newcommand{\LL}{\mathcal L}
\DeclareMathOperator{\Gr}{\operatorname{\mathbf{G}}}
\newcommand{\GG}{\mathcal{G}}
\newcommand{\UU}{\mathcal{U}}
\newcommand{\FF}{\mathcal{F}}
\newcommand{\ra}{\rightarrow}
\newcommand{\EE}{\mathcal{E}}
\newcommand{\lp}{(}
\newcommand{\rp}{)}
\renewcommand{\det}{\operatorname{det}}

\DeclareMathOperator{\lw}{lw}
\DeclareMathOperator{\todd}{todd}
\DeclareMathOperator{\ext}{ext}
\DeclareMathOperator{\ch}{ch}
\DeclareMathOperator{\Quot}{Quot}
\DeclareMathOperator{\Hom}{Hom}
\DeclareMathOperator{\Ext}{Ext}
\DeclareMathOperator{\Tor}{Tor}

\numberwithin{equation}{section}

\newtheorem{thm}[equation]{Theorem}
\newtheorem{lem}[equation]{Lemma}
\newtheorem{prop}[equation]{Proposition}
\newtheorem{conj}[equation]{Conjecture}

\theoremstyle{remark}
\newtheorem{remark}[equation]{Remark}

\title[Representations in Strange Duality]{Representations in Strange Duality: Hilbert Schemes paired with higher rank spaces}
\author{Drew Johnson}
\email{jared.johnson@math.ethz.ch}
\address{ETH---Z\"{u}rich, Switzerland}

\date{\today}

\begin{document}

\maketitle

\begin{abstract}
	We examine a sequence of examples of pairs of moduli spaces of sheaves on $\PP^2$ where Le Potier's strange duality is expected to hold. One of the moduli spaces in these pairs is the Hilbert scheme of two points. We compute the sections of the relevant theta bundle as a representation of $SL(X)$, where $\PP^2=\PP(X)$. For the higher rank space, we construct a moduli space using the resolution of exceptional bundles from Coskun, Huizenga, and Woolf \cite{coskun_effective_2014}. We compute a subspace of the sections of the theta bundle which is dual to the sections on the Hilbert scheme.
	
	In the second part, we use a result from Goller and Lin \cite{goller_rank1_2019} to rigorously apply the ``finte Quot scheme method", introduced in Marian and Oprea \cite{marian_level-rank_2007} and Bertram, Goller, and Johnson \cite{bertram_potiers_2016}. This requires us to prove that the kernels in appearing in the Quot scheme have the appropriate resolution by exceptional bundles. 
	
\end{abstract}

\section{Introduction}
\subsection{Strange Duality for \texorpdfstring{$\PP^2$}{P2}} Strange duality is a conjectural perfect pairing between spaces of sections of theta bundles on moduli spaces of sheaves with orthogonal invariants over a fixed variety $Y$.

We focus on the case of $\PP^2$ in this paper. If  $h=c_1(\OO(1))$ is the hyperplane class, we write elements $a+bh+ch^2 \in H^*(\PP^2,\mathbb Q)$ as triples $(a,b,c)$. 
Let $e,f$ 
be cohomology classes that are orthogonal with respect to the bilinear form 
\[
\chi(e\cdot f) = \int_{\PP^2} e\cdot f \cdot \todd(\PP^2),
\]
The Hirzebruch-Riemann-Roch Theorem implies that for $E$ and $F$ coherent sheaves on $Y$, the Chern characters pair as
\[
\chi\left(\ch(E)\cdot \ch(F)\right) = \chi(E \otimes F) := \sum (-1)^i h_i(E\otimes F).
\]

Let $M(e)$ and $M(f)$ be moduli spaces of sheaves of the respective Chern characters. We will say more specifically which spaces we want to use later. 

We say that $e$ and $f$ are \emph{candidates for strange duality}, if, in addition to the orthogonality condition, we have
\begin{enumerate}
	\item $H^2(E\otimes F)=0$, and $\Tor^1(E,F)=\Tor^2(E,F) = 0$  for $(E,F) \in M(e) \times M(f)$ away from a subset of codimension 2.
	\item There exists some pair $(E,F) \in M(e) \times M(f)$ so that $H^0(E \otimes F)=0$.
\end{enumerate}  
These conditions ensure that the  ``jumping locus''
\[
\Theta = \{\, (E,F) \mid \Hom(E, F) \neq 0 \,\} \subset M(e) \times M(f)
\]
has the structure of a Cartier divisor (see \cite{le_potier_dualite_2005}, \cite{scala_dualite_2007}, \cite{danila_resultats_2002}), \cite{marian_tour_2008}).

We also have divisors 
\[
\Theta_F = \{[ E]: H^0( E \otimes F) >0\} \subset M(e), \;\;\;\; \Theta_{E} = \{[F]:H^0(E \otimes F) >0\} \subset M(f). 
\] 
In the case that $M(e)$ has a universal family $\EE$, the class of $\Theta_F$ is given by
\begin{equation}
 \det( Rp_*(\EE\otimes  q^*F))^\vee \label{eq:theta-class}
\end{equation}
where $p$ and $q$ are the projections on $M(e) \times \PP^2$ \cite{marian_tour_2008}. The class of $\Theta_E$ is computed similarly .

The Picard group of moduli spaces of sheaves on $\PP^2$ is discrete. Since  $\Theta_F$ varies continuously as $F$ varies in $M(f)$, it follows that the $\Theta_F$ are all linearly equivalent; we refer to the associated line bundle as $\mathcal O(\Theta_f)$ and similarly for $\mathcal O(\Theta_{e})$.
Then, one can  check that the  line bundle associated to $\Theta$ satisfies
\[
{\mathcal O}_{M(e)\times M(f)}(\Theta) = {\mathcal O}_{M(e)}(\Theta_f)\boxtimes {\mathcal O}_{M(f)}(\Theta_{e})
\]
Thus
$$H^0\big(M(e) \times M(f), {\mathcal O}(\Theta)\big) = H^0\big(M(e),{\mathcal O}(\Theta_f)\big) \otimes 
H^0\big(M(f),{\mathcal O}(\Theta_{e})\big),$$
so a section defining $\Theta$ determines a map (up to a choice of a scalar):
\begin{equation}
\operatorname{SD}_{e,f}: H^0\big(M(f),{\mathcal O}(\Theta_{e})\big)^\vee \rightarrow H^0\big(M(e),{\mathcal O}(\Theta_f)\big) \label{eq:sd}
\end{equation}
One could interpret $\operatorname{SD}_{e,f}$ as taking the hyperplane $H_F$ in $H^0\big(M(f),{\mathcal O}(\Theta_{e})\big)$ of sections vanishing at $[F] \in M_S(f)$ to the section (up to scaling) $\Theta_F \in H^0\big(M(e),{\mathcal O}(\Theta_f)\big)$. 

\begin{conj}[Le Potier's Strange Duality] $\operatorname{SD}_{e,f}$ is an isomorphism.
\end{conj}

\begin{remark}
Strange duality conjectures also exist for abelian and K3 surfaces, but some modifications must be made by fixing a determinant and/or determinant of the Fourier Mukai transform. See for example \cite{marian_sheaves_2008}, \cite{marian_generic_2010}, and \cite{marian_strange_2014}. This is also true of the classical strange duality for curves \cite{marian_level-rank_2007}.
\end{remark}


\subsection{Representations}
Let $X$ be a 3 dimensional vector space and $\PP(X) \cong \PP^2$ be the projective space of lines in $X$. Then $SL(X)$ acts on $\PP(X)$. Then $SL(X)$ also acts on moduli spaces on $\PP(X)$ by pull back. Using the description below \eqref{eq:sd}, we can see that $SD$ is $SL(X)$ equivariant: for $g \in SL(X)$ we have $g(H_{F}) = H_{g*F}$, and $g(\Theta_F) = \Theta_{g^*F}$. Strange duality is viewed as a map of $SL(X)$-modules also in \cite{danila_resultats_2002}. Our goal will be to compute spaces of sections as $SL(X)$ representations.



\subsection{Results}
 We will consider a sequence of examples on $\PP(X) = \PP^2$.
For any $m \ge 1$, let $e=(m+1, 2m+1, -2m-\frac12)$, and $f=(1,0,-2)$. The moduli space $M(f)$ is the Hilbert scheme of $2$ points,  $\PP(X)^{[2]}$. We will discuss this space and its theta bundle in Section \ref{sec:hilb}. 

In Section \ref{sec:kron}, we will address $M(e)$.
In \cite{coskun_effective_2014}, Coskun, Huizenga, and Woolf show how to construct nice resolution for a general semistable sheaf of a fixed Chern character. We recall this construction in Section \ref{sec:beil}. Our class $e$ was chosen so that the resolution constructed this way has a particular form \eqref{eq:Eres}, which then expresses $E$ as a extension \eqref{eq:Eext}. For the space $M(e)$, instead of working with the moduli space of semistable sheaves, we will instead use a moduli space of these extensions. The advantage is that we can see explicitly how this space is built up from $X$. We will argue in \cref{sec:relationship} that this space is birational to $M^{ss}(e)$, the moduli space of Gieseker semistable sheaves.

We will prove (see the next section for notations):

\begin{thm} \mbox{} \label{thm-sec}
	\begin{enumerate}[label=(\alph*)]
		\item
	The space of sections $H^0(\PP(X)^{[2]}, \OO(\Theta_{e}))$ is equal to 
	\[
	  (S^{3m+1,m+1} + S^{3m-1,m+3}  + \dots + S^{2\lceil \frac m2\rceil +m +1,2\lfloor\frac m2\rfloor+m+1}) (X^\vee)
	\]
	as an $SL(X)$-module.
	
	\item
	The space of sections $H^0(M(e), \OO(\Theta_f))$ contains a sub-$SL(X)$-module 
	\[
	 (S^{3m+1,m+1} + S^{3m-1,m+3}  + \dots + S^{2\lceil \frac m2\rceil +m +1,2\lfloor\frac m2\rfloor+m+1}) (X)
	\] \label{thm-sec-kron}
\end{enumerate}
\end{thm}

The case $m=1$ was given as an example in \cite{johnson_two_2016}. 

In Section \ref{sec:quot}, we recall the ``finite Quot scheme method" of approaching strange duality. This method was first employed by Marian and Oprea in \cite{marian_level-rank_2007} to give a proof of strange duality for curves. Later, Bertram, Goller, and the author used this method to investigate strange duality for surfaces \cite{bertram_potiers_2016}, see also \cite{johnson_universal_2017}.
Using a recent result of Goller and Lin \cite{goller_rank1_2019}, it is now possible to apply this method rigorously. Some work is still required, which we do in \cref{sec:quot}. We obtain: 
\begin{thm} \label{thm:inj}
	Let $e$, $f$ be as above. Then $SD_{e,f}$ is injective for $m \le 5$.
\end{thm}
 
In order to upgrade this statement to ``isomorphism", one needs to better understand the space of sections $H^0(M(e), \OO(\Theta_f))$. For example, computing its dimension would be sufficient. 

\subsection{Acknowledgments} The author thanks Thomas Goller and Yinbang Lin for helpful discussions during the work on this paper. The author was supported by the Swiss National Science Foundation through grant SNF-200020-18218.

\section{Preliminaries}
\subsection{Notation and Conventions}
We write $\lceil \cdot \rceil$ and $\lfloor \cdot \rfloor$ for the ceiling and floow functions (round up, respectively down, to the nearest integer).

For a vector space $X$, we write $\PP(X) = \operatorname{Proj} (\operatorname{Sym}^{\bullet} X^\vee)$, so $\PP(X)$ parameterizes non-zero points of $X$, up to a scalar, and $H^0(\PP(X),\OO(1)) = X^\vee$.  Similarly, for a vector bundle $\FF$ on a variety $Y$,  the projective bundle $\pi: \PP(\FF) \ra Y$ parameterizes points in the total space of $\FF$, up to the action of $\CC^\times$ on the fibers, and $\pi_* (\OO_{\PP(\FF)}(1)) = \FF^\vee$.

\subsection{Schur functors and partitions}
 A partition is weakly decreasing list of positive integers $\lambda = \lp\lambda_1,\lambda_2,\dots, \lambda_k\rp $.  We use exponents to indicate repeated parts, hence $\lp2^m,1\rp  = \lp2,2,\dots, 2,1\rp $ (with the $2$ repeated $m$ times). We write $\ell(\lambda) = k$ for the \emph{length} of the partition, and $|\lambda| = \sum_i \lambda_i$. If $|\lambda| = n$, we write $\lambda \vdash n$ and say that $\lambda$ is a partition of $n$. We we also use the notation $2\lambda$ for the partition $\lp2\lambda_1, 2\lambda_2, \dots, 2\lambda_k\rp $. It is often convenient to use the convention $\lambda_j =0$ for $j > k$. If $\mu = \lp\mu_1, \dots, \mu_p\rp $, then we define $\lambda + \mu$ to be the partition $\lp\lambda_1 + \mu_1, \lambda_2+\mu_2, \dots, \lambda_q + \mu_q\rp $, where $q = \operatorname{max}(k,p)$.
 
 Representations of $GL(n)$ and $SL(n)$ are classified by Schur functors. Schur functors $S^\lambda$ are functors (indexed by partitions $\lambda$) from the category of vector spaces (or vector bundles) to itself. We usually omit the parentheses when writing explicit partitions in the superscript of $S$.
 
 \begin{thm}
 	Let $U$ be a vector space of dimension $n$. 
 	
 	\begin{enumerate}
 		\item $S^{\lambda}U$ is an irreducible representation of $GL(U)$ and $SL(U)$.
 		\item Any irreducible representation of $GL(U)$ can be written as $S^\lambda U \otimes (\det U^\vee)^c$ for some partition $\lambda$ and integer $c \ge 0$. 
 		\item $S^\lambda U \otimes (\det U^\vee)^c \cong S^\mu U \otimes (\det U^\vee)^d$ as $GL(U)$ representations if and only if $\lambda + \lp d^n\rp  = \mu + \lp c^n\rp $. \label{item:schur-unique}
 		\item Any irreducible representation of $SL(U)$ can be written as $S^\lambda(U)$ for some partition $\lambda$.
 	\end{enumerate}.
\end{thm}

In light of \eqref{item:schur-unique} above, if $\lambda = \lp a_1, \dots, a_k\rp $ is a sequence of weakly decreasing integers with $a_k < 0$, it makes sense to  define $S^\lambda U = S^{\lambda + (-a_k)^n} U \otimes (\det U^\vee)^{-a_k}$. (We will not call such a sequence a partition.)

The following are basic facts that we will make use of without further comment.

\begin{thm}
 		Let $\lambda$ and $\mu$ be finite, weakly decreasing sequences of intergers and $U$ be a vector space of dimension $n$.
 		\begin{enumerate}

 		\item If $\ell(\lambda) > n$, then $S^\lambda(U) = 0$.
 		\item $S^\lambda U \otimes (\det U)^c = S^{\lambda + \lp c^n\rp }U$ for any integer $c$.
 		\item $S^\lambda U \cong S^\mu U$ as $SL(U)$ representations if and only if $\lambda = \mu + \lp c^n\rp $ for some integer $c$.
 		\item  $\operatorname{Sym}^k \cong S^k U $.

 		\item $\bigwedge^k U \cong S^{1^k}U $.
 		\item $S^{-1} U = U^\vee$.
 		\item $S^\lambda(U)^\vee = S^\lambda(U^\vee)$.
 		\item $S^\lambda(U \otimes \det U^c) = S^\lambda U \otimes (\det U)^{c|\lambda|}$
 	\end{enumerate}
 
 \end{thm}

\subsection{The Littlewood-Richardson Rule}
If $\mu$ and $\lambda$ are partitions, then the tensor product $S^\mu U \otimes S^\lambda U $ can be decomposed as $\bigoplus_{\nu} S^\nu U ^{\oplus c^\nu_{\mu,\lambda}}$, where the sum is over $\nu \vdash (|\lambda| + |\mu|)$. The coefficients $c^\nu_{\mu,\lambda}$ have a combinatorial description known as the Littlewood-Richardson rule, and are referred to as Littlewood-Richardson coefficients.

We will use straightforward applications of this rule without further comment, however, we will need to make some non-trivial calculations in the proof of \cref{summand}. Instead of recalling the definitions and the rule, we refer the reader to any standard source, for example \cite{leeuwen_littlewood_1999}.

\subsection{Schur functions} \label{sec:schur-functions}
The following will be used in the proof of \cref{sln-inv}.

Given a sequence of variables $x=(x_1, \dots, x_k)$, the Schur functions $s_\lambda(x)$  form a (linear) basis for the symmetric polynomials in $x_1, \dots, x_k$. They are labeled by partitions, just as the Schur functors $S^\lambda$. A semistandard Young tableau of shape $\lambda$ with alphabet $x$ is an table of a certain type with entries from $x$, satisfying certain rules. Its \emph{content} is the product of its entries. The Schur function $s_\lambda(x)$ can be defined as the sum of the contents of all semistandard Young tableaux of shape $\lambda$. We refer the reader to any standard source for full definitions.

 If $U$ has dimension $k$, then the character of $S^\lambda(U)$ is $s_\lambda(x_1, \dots, x_k)$. That is, after picking a basis, the  trace of the action on $S^\lambda(U)$ by the element $\operatorname{diag}(x_1, \dots, x_k) \in GL(U)$ is $s_\lambda(x)$. In fact, the character determines the representation. Hence, the Schur functions obey the same Littlewood-Richardson rule: $s_\mu(x) s_\lambda(x) = \sum_\nu c^\nu_{\mu, \lambda} s_\nu(x)$.  

We also mention the monomial basis for the symmetric functions. Given a partition $\lambda = \lp a_1, \dots, a_k\rp $ (padded with 0's if necessary), let 
\[
m_\lambda(x) = \sum_{\sigma \in \Sigma_k}  \prod_i x_{\sigma(i)}^{a_i} 
\]
Here $\Sigma_k$ is the symmetric group on $1, \dots, k$. Given an explicit symmetric polynomial, it is easy to write it as a sum of monomial symmetric polynomials.

Partitions can be ordered lexicographically, which we call the weight. If one expresses $s_\lambda(x)$ as a linear combination of terms of $m_\alpha(x)$, the highest weight term will be $\alpha = \lambda$.

This leads to the following algorithm. Suppose one is given a symmetric polynomial $f(x)$, expressed as a sum of monomial symmetric functions, and wishes to express it as a sum of Schur  functions. Find the highest weight partition $\alpha$ in $f(x)$ with non-zero coefficient $b_\alpha$. Then $f(x) - b_\alpha s_\alpha(x)$ is a symmetric function with smaller highest weight. One can continue inductively.



\subsection{Sections of Vector Bundles on Grassmannians}
Finally, we recall this very useful result.

\begin{prop}[\cite{edidin_grassmannians_2009}] \label{prop:grass-sec}
	Let $\Gr(k,U)$ be the Grassmannian of $k$-dimensional subspaces of an $n$-dimensional vector space $U$. Let $\SS$ and $\QQ$ be the universal sub- and quotient bundles, respectively. 
	
	For any partition $\lambda$ with $n-k$ or fewer parts,  we have $H^0(\Gr(k,U), S^\lambda(\QQ)) = S^\lambda U$, and the higher cohomology vanishes. If $\lambda$ is a sequence with negative entries, then $H^0(\Gr(U,k), S^\lambda(\QQ)) = 0$.
	
	For any partition $\lambda$ with $k$ or fewer parts, we have $H^0(\Gr(U,k), S^\lambda(\SS^\vee)) = S^\lambda U^\vee$, and the higher cohomology vanishes. If $\lambda$ is a sequence with negative entries, then $H^0(\Gr(U,k), S^\lambda(\SS^\vee)) = 0$.

\end{prop}

\begin{remark}
	The Pl\"{u}cker line bundle $\OO_{\Gr(U,k)}(1)$ is $\bigwedge^k \QQ$, or in terms of $\SS$,  this is isomorphic to $\det U \otimes \bigwedge^{n-k} \SS^\vee$. \cref{prop:grass-sec} says that it has sections $\bigwedge^k U$. 
\end{remark}

\section{The Hilbert Scheme} \label{sec:hilb}

When $f = 1 - 2h^2$, then $M(f) = \PP(X)^{[2]}$, the Hilbert scheme parameterizing dimension 0, length 2 subschemes (or rather, their ideal sheaves).

A length 2 subscheme of $\PP(X)$ determines a line, and hence we have a map $\PP(X)^{[2]} \rightarrow \PP(X^\vee)$. A fiber of this map over a point $[\ell]$ is simply the Hilbert scheme of the line $\ell$, which is isomorphic to $\PP(S^2\ell^\vee)$ via the map that takes a degree two polynomial on $\ell$ to its roots. So we can view $\PP(X)^{[2]}$ as a projective bundle over $\PP(X^\vee)$. Let $\LL$ be the (cone over the) universal line over $\PP(X^\vee)$, where $\PP(X^\vee)$ is viewed as the moduli space of lines in $\PP(X)$. So we have $\PP(X)^{[2]} \cong \PP(S^2\LL^\vee)$.

We will be interested in computing the spaces of sections of line bundles on $\PP(S^2\LL^\vee)$. 
The group of line bundles is generated by the Serre twisting sheaf $\OO_{\PP(S^2 \LL^\vee)}(1)$ of the projective bundle, and the pullback of the Serre twisting sheaf from the base $\pi^*\OO_{\PP(X^\vee)}(1)$. 
\begin{lem} \label{lem:Gsec}
	The space of sections \[
	\Gamma_{d,f} := H^0(\PP(S^2 \LL^\vee), \OO_{\PP(S^2 \LL^\vee)}(d) \otimes \pi^*\OO_{\PP(X^\vee)}(f))
	\] is calculated as follows:
	\begin{enumerate}
	
	\item If $f$ is even, and $d \le f \le 2d$, then
	\[
	 \Gamma_{d,f} = \sum_{\substack{\lambda \vdash (f-d) \\ \ell(\lambda)\le2}} S^{2\lambda}(X^\vee).
	\] \label{1a}
	
	\item If  $f$ is odd, and $d+1 \le f \le 2d-1$, then 
	\[
	\Gamma_{d,f} = \sum_{\substack{\lambda \vdash (f-d-1) \\ \ell(\lambda)\le2}} S^{2\lambda + \lp 1,1\rp }(X^\vee).
	\] \label{1b}
	
	\item If $f \ge 2d$, then
	\[
	\Gamma_{d,f} = \sum_{\substack{\lambda \vdash d \\\ell(\lambda)\le2} } S^{2\lambda + \lp f-2d,f-2d\rp }(X^\vee).
	\] \label{1c}
	
	\item If $f < d$, then $\Gamma_{d,f} = 0$. \label{1d}
\end{enumerate}
\end{lem}
In the proof, we will need the following, which is due to Littlewood. 
\begin{prop}
	Let $U$ be a vector space or a vector bundle. Then
	\[
	S^{m}(S^2(U)) = \bigoplus_{\lambda \vdash m} S^{2\lambda}(U).
	\]
	where the sum is over partitions $\lambda$ of $m$ with length less than or equal to the dimension of $U$.
\end{prop}

\begin{proof}[Proof of Lemma \ref{lem:Gsec}]
First, from the Euler sequence we see that $\LL = \Omega_{\PP(X^\vee)}(1)$, hence $\det \LL  =\omega_{\PP(X^\vee)} (2)=\OO_{\PP(X^\vee)}(-1)$. We also have $S^2\LL  = S^2\LL^\vee \otimes (\det \LL)^{\otimes 2}$. We then calculate the pushforward:
\begin{align*}
\pi_*(\OO_{\PP(S^2\LL^\vee)}(d) \otimes \pi^*\OO_{\PP(X^\vee)}(f)) &= S^d(S^2 \LL) \otimes \OO_{\PP(X^\vee)}(f)\\
&=S^d(S^2\LL^\vee \otimes (\det \LL)^{\otimes 2}) \otimes \det (\LL^\vee)^f \\
&= S^d(S^2 \LL^\vee) \otimes \det(\LL^\vee)^{f-2d} \\
&= \bigoplus_{\lambda \vdash d} S^{2\lambda + \lp f-2d, f-2d\rp } \LL^\vee
\end{align*} 
Now \eqref{1c} and \eqref{1d} follow from Proposition \ref{prop:grass-sec}.

If $f=2k$ is even  this is equal to
\[
 \bigoplus_{\lambda \vdash d} S^{2(\lambda + \lp k-d, k-d\rp )} \LL^\vee
\]
If $d \le f \le 2d$, then $k-d$ is negative. Notice that if you take the set of partitions of $d$ of length at most 2, add $\lp k-d,k-d\rp $ to each one, and throw away those with negative entries, you get precisely the partitions of $f-d$ of length at most 2. So the \eqref{1a} follows. Part \eqref{1b} follows similarly.
\end{proof}

In order to avoid having to identify the universal bundle on $\PP(S^2 \LL^\vee)$, we recall the identification of the theta bundle on the Hilbert scheme in \cite{bertram_potiers_2016}. The Picard group of the Hilbert scheme of $n$ points is generated by divisors $H_n$ and $\frac{B}{2}$, where $H_n$ is the locus of subschemes meeting a fixed hyperplane, and $B$ is the locus of non-reduced subschemes (in other words, where some of the points have collided).

We then have 
\begin{prop}
	Let $e = (r,c, s)$, where $s$ is chosen so that $\chi(e, (1,0-n)) = 0$. Then the theta divisor on $\PP(X)^{[n]}$ is 
	\[
	\Theta_e = c \cdot H_n - r \frac{B}{2}.
	\]
\end{prop}

We now need to see how to change from the basis $\left< \OO(H_2), \OO(-\frac B2) \right>$ to the basis $\left<\OO_{\PP(S^2\LL^\vee)}(1), \pi^* \OO_{\PP(X^\vee)}(1) \right>$.

To get information about spaces of sections, we recall a result from \cite{arcara_minimal_2013}. We state it only for the case of $n=2$ points. Let $\Gr_k := \Gr(H^0(\OO_{\PP(X)}(k)),2)$ be the Grassmannian of dimension 2 quotients of $H^0(\OO_{\PP(X)}(k))$. For $k\ge1$, there is a map $\phi_k: \PP(X)^{[2]} \rightarrow \Gr_k$ given by taking a zero dimensional subscheme $Z$ to the quotient $H^0(\OO(k)) \ra H^0(\OO(k)|_Z)$. Let $\OO_{\Gr_k}(1)$ be the Pl\"{u}cker line bundles, and let $D_k = \phi_k^*\OO_{\Gr_k}(1)$.
\begin{thm}[Proposition 3.1, \cite{arcara_minimal_2013}]
 The class of $D_k$ is given by
 \[
 D_k = kH_2 - \frac{B}{2}.
 \]
 
 $D_k$ is very ample for $k \ge 2$.
\end{thm}

First, we see that $\phi_1$ is just the map giving the projective bundle $\PP(S^2 \LL^\vee)$. We conclude that $\pi^*\OO_{\PP(X^\vee)}(1)$ corresponds to $\OO(H_2 - \frac{B}{2})$. 

 A section $\sigma$ of the Pl\"{u}cker bundle on $\Gr(H^0(\OO(2)),2)$ can be given by picking $s_1$, $s_2 \in S^2 V^\vee$. Then, at any  dimension 2 quotient $q: H^0(\OO(2)) \ra Q$, the section $\sigma$ evaluates to $q(s_1) \wedge q(s_2) \in \bigwedge^2 Q$. Let us examine the vanishing locus of $\sigma$ when pulled back via $\phi_2$ to $\PP(S^2 \LL^\vee)$. Let $\ell \in \PP(X^\vee)$, and $f \in S^2\ell^\vee$ a point in the fiber over $\ell$. We see that the corresponding 2 dimensional quotient in the Grassmannian is $S^2\ell^\vee/\left<f\right>$. The wedge $s_1 \wedge s_2$ vanishes in the quotient if and only if  $s_1|_\ell$ and $s_2|_\ell$ are linearly dependent mod $f$. So we must have either $a s_1|_\ell + b s_2|_\ell =f$ or $a s_1|_\ell + b s_2|_\ell =0$ for some $a,b \in \CC$. In the former case, the vanishing locus is the hyperplane $\operatorname{span}(s_1|_\ell, s_2|_\ell) \subset S^2\ell^\vee$. In the later, the section vanishes on the whole fiber.
 
 Let us now view $S^2 X^\vee$ as the space of symmetric bilinear forms on $X$. Let $x,y,z$ be a basis for $X^\vee$. Pick $s_1 = x \otimes x$ and $s_2 = y \otimes y$. For any $\ell$ defined by an equation of the special form $ax+by$, we see that 
 \[
 a^2 (x|_\ell \otimes x|_\ell) = ax|_\ell \otimes ax|_\ell = (-by|_\ell) \otimes (-by|_\ell) = b^2 (y|_\ell \otimes y|_\ell)
 \]
 so the section vanishes on the whole fiber over $\ell$. For any other $\ell$, $x|_\ell \otimes x|_\ell$ and $y|_\ell \otimes y|_\ell$ are independent in $S^2 \ell^\vee$, so we get a hyperplane.

 Let $L = ax + by +cz \in X^\vee$, and let $\ell$ be the corresponding line. Associate to $L$ the point $p_L:=(0, -c, b) \otimes (-c^2,0,ac) \in S^2 \ell$ (here we view $S^2\ell$ as a quotient of $\ell \otimes \ell$). If $ c \neq 0$, then $x|_\ell \otimes x|_\ell$ and $y|_\ell \otimes y|_\ell$ are independent. Notice that both  $x|_\ell \otimes x|_\ell$ and $y|_\ell \otimes y|_\ell$ vanish on $p_L$, so they span the hyperplane defined by $p_L$. On the other hand, if $c=0$, then $p_L = 0$ and every point of $S^2\ell^\vee$ vanishes on $p_L$. So the association $L \mapsto p_L$ describes the same set as the vanishing locus of $\phi_2^*\sigma$.
 
 The association $L \mapsto p_L$ is 3-homogeneous, and generically it gives a hyperplane in the fiber. Hence it corresponds to the line bundle $\OO_{\PP(S^2\LL^\vee)}(1) \otimes \pi^* \OO_{\PP(X^\vee)}(3)$. 
 
Now, by linear algebra, the map $(c,r) \mapsto (d,f)$ so that $\OO(cH_2 - r \frac B2)$ corresponds to $\OO_{\PP(S^2\LL^\vee)}(d) \otimes \pi^* \OO_{\PP(X^\vee)}(f)$ is given by the matrix
 \[
  \begin{bmatrix}1 & -1 \\ 2 & -1\end{bmatrix}
 \]
So the theta bundle $(2m+1)h_2 - (m+1)\frac{B}{2}$ corresponds to $\OO_{\PP}(m) \otimes \pi^* \OO_{\PP(X^\vee)}(3m+1)$. Now we invoke Lemma \ref{lem:Gsec}, and conclude that the space of sections of the theta bundle is 
\[
(S^{3m+1,m+1} + S^{3m-1,m+3} + S^{3m-3,m+5} + \dots + S^{2\lceil \frac m2\rceil +m +1,2\lfloor\frac m2\rfloor+m+1}) (X^\vee).
\]
thus proving the first part of \cref{thm-sec}.

\section{Resolutions by Exceptional Bundles} \label{sec:beil}
This section gives motivation for our our choice of moduli space $M(e)$ in the next section. In \cref{sec:relationship}, we will use the results here to argue that $M(e)$ is birational to the  moduli space $M^{ss}(e)$.  We omit a discussion of Gieseker stability here, as it can be found in many places and is not essential to this paper. 

We do need to recall the resolution by exceptional sheaves from \cite{coskun_effective_2014}.   A stable vector bundle $E$ on $\PP^2$ is an \emph{exceptional bundle} if $\Ext^1(E,E)=0$. The \emph{slope} of a vector bundle $E$ on $\PP^2$ is given by $\mu(E) = c_1(E)/\operatorname{rank}(E)$ (where $c_1(E)$ is viewed as an integer). For a rational number $\alpha$, there is at most one exceptional bundle of slope $\alpha$. If there is such a bundle, we call it $E_\alpha$ and say that $\alpha$ is an \emph{exceptional slope}.  For example, all integers are exceptional slopes since the line bundles $\OO(n)$ are exceptional bundles.

Let $r_\alpha$ be the rank of $E_\alpha$, which is equal to the denominator of $\alpha$. Let $\Delta_\alpha = \tfrac{1}{2}(1-\tfrac{1}{r_\alpha^2})$ be the discriminant of $E_\alpha$. The set of exceptional slopes $\mathcal{X}$ is in bijection with the dyadic integers via a function $\varepsilon \colon \mathbb{Z}[\tfrac{1}{2}] \to \mathcal{X}$ defined inductively by $\varepsilon(n)=n$ for $n \in \mathbb{Z}$ and by setting
\[
\varepsilon\left( \frac{2p+1}{2^{q+1}}\right) = \varepsilon\left( \frac{p}{2^q}\right). \varepsilon\left( \frac{p+1}{2^q}\right),
\]
where the dot operation on exceptional slopes is defined by $\alpha.\beta = \tfrac{\alpha+\beta}{2} + \tfrac{\Delta_\beta - \Delta_\alpha}{3+\alpha-\beta}$. Each dyadic integer can be written uniquely as $\frac{2p+1}{2^{q+1}}$, so the $p$ and $q$ in the equation are uniquely determined, and we call the equation the \emph{standard decomposition} of the exceptional slope $\varepsilon\left( \tfrac{2p+1}{2^{q+1}}\right)$.

Let $g \in H^*(\PP^2,\mathbb Q)$ be a Chern character. To find the exceptional slopes for resolving general sheaves in $M^{ss}(g)$, one needs the \emph{corresponding exceptional slope} $\gamma$ of $g$. This is obtained by first computing
\[
\mu_0 = -\tfrac{3}{2}-\mu + \sqrt{\tfrac{5}{4}+\mu^2-\tfrac{2\ch_2(g)}{r}},
\]
where $r$ is the rank of $g$ and $\mu := c_1(g)/r$ is the slope. Then $\gamma$ is the unique exceptional slope satisfying $|\mu_0-\gamma| < x_{\gamma}$, where $x_{\gamma} = \tfrac{3}{2}-\sqrt{\tfrac{9}{4}-\tfrac{1}{r_\gamma^2}}$.

\begin{prop}[\cite{coskun_effective_2014}]\label{beil-res} Let $g$ be a Chern character, let $\gamma$ be the corresponding exceptional slope to $g$, and let $\gamma = \alpha.\beta$ be the standard decomposition of $\gamma$. If $\chi(g \cdot \ch E_\gamma) \ge 0$, then the general $G \in M(g)$ has a resolution
		\[
		0 \to E_{-\alpha-3}^{m_1} \to E_{-\beta}^{m_2} \oplus E_{-\gamma}^{m_3} \to G \to 0,
		\]
		where $m_1 = -\chi(G \otimes E_\alpha)$, $m_2 = -\chi(G \otimes E_{\alpha.\gamma})$, $m_3 = \chi(G \otimes E_\gamma)$.
\end{prop}

 We want to apply \cref{beil-res} with $g = e=(m+1,2m+1, -2m-\frac12)$. We claim that the corresponding exceptional slope is $\gamma = -\frac12$. We have $x_{-\frac12} = \frac32 - \sqrt{2} \approx 0.08579$.

The inequality $|\mu_0 - \gamma| < x_{-\frac12}$ can be verified by a straightforward computation for any small values of $m \ge 1$. So we focus on the cases $m \gg 1$.

We first observe that $\mu = \frac{2m+1}{m+1}$ and $-\frac{2\ch_2(g)}{r} = \frac{2(2m + \frac12)}{m+1}$ are both increasing functions for $m \ge 1$, with limits $2$ and $4$, respectively. Thus, for any $m \ge N \ge 1$, we have
\[
-\frac32 - 2 + \sqrt{\frac54 + \left(\frac{2N+1}{N+1}\right)^2 + \frac{2(2N + \frac12)}{N+1}} \le \mu_0(m) \le -\frac32 - \frac{2N+1}{N+1} + \sqrt{\frac54 + 4 + 4}
\]
When $N = 22$, this inequality gives
\[
  -0.50876 < \mu_0 < -0.41514
\] 
so we see that the associated exceptional slope is $-\frac{1}{2}$ for all $m \ge 22$, which just leaves just 21 cases that can easily be checked by a computer.

Hence $\alpha = -1$, $\gamma = -\frac12$, and $\beta = 0$. So $E_{-\alpha-3} = \OO(-2)$, $E_{-\gamma} = T(-1)$, and $E_{-\beta} = \OO$. 
A calculation with the Hirzebruch-Riemann-Roch theorem shows that $\chi(e \cdot \ch E_{-\frac12}) =1 >0$. Further, one calculates $m_1$, $m_2$, and $m_3$ to obtain:
\begin{prop} \label{gen-res}
	Let $n=2m-1$. A general semistable sheaf $ E$ with Chern character $e$ has a resolution
	\begin{equation}
	 0 \ra \OO(-2)^m \ra T(-1) \oplus \OO^{n} \ra  E \ra 0 \tag{$\star$} \label{eq:Eres}
	\end{equation}
\end{prop}

We will refer to such a resolution as a \ref{eq:Eres}-resolution. We will also need a slight strengthening of \cref{gen-res}:

\begin{prop} \label{gen-gen}
	A general semistable sheaf $E$ with Chern character $e$ has a general \ref{eq:Eres}-resolution.
	
	More precisely, for any proper  closed subset $Z \subset M^{ss}(e)$, the set 
	\[
	\left\{\phi \in 	\Hom(\OO(-2)^m, T(-1) \oplus \OO^{n}): [\ker \phi] \in Z \right\}
	\]
	is contained in a proper (Zariski) closed subset of $\Hom(\OO(-2)^m, T(-1) \oplus \OO^{2m-1})$.
\end{prop}
\begin{proof}
	In the proof of Proposition 5.3 in \cite{coskun_effective_2014}, we saw that a general map $\OO(-2)^m \ra T(-1) \oplus \OO^{2m-1}$ is injective and has a stable cokernel. Let $U \subset \Hom(\OO(-2)^m , T(-1) \oplus \OO^{2m-1})$ be a Zaraski open set consisting of such maps. There is a morphism $C:U \ra M^{ss}(e)$ taking a map to its cokernel. Then $C^{-1}(Z)$ is a closed subset of $U$ and hence it closure is a proper closed subset of $\Hom(\OO(-2)^m , T(-1) \oplus \OO^{2m-1})$.
\end{proof}

\section{The Higher Rank Space} \label{sec:kron} 
\subsection{The parameter space} \label{sec:the-mod}
We set $n=2m-1$, as in the previous section.
For our moduli space, we will use set of all coherent sheaves $E$ that fit in an exact sequence
\begin{equation}
0 \ra T(-1) \rightarrow E \rightarrow F \rightarrow \ra 0 \label{eq:Eext}
\end{equation}
where $F$ is the cokernel\footnote{This $F$ is \emph{not} the same as the $F$ that is orthogonal to $E$, as in the introduction.} of an injective map $\OO(-2)^m \ra \OO^{n}$. We remark that for a general \ref{eq:Eres}-resolution, the map $\OO(-2)^m \ra \OO^{n}$ will be injective.
Then one can form the diagram
\[
\xymatrix{
&&T(-1) \ar[r]^{=} \ar[d] & T(-1) \ar[d] & \\
0 \ar[r] & \OO(-2)^m \ar[r] \ar[d] & T(-1) \oplus \OO^n \ar[r] \ar[d] & E \ar[r] \ar[d] &0 \\
0 \ar[r] & \OO(-2)^m \ar[r] &  \OO^n \ar[r] & F \ar[r] &0
}
\] where the right column gives an extension of the form \eqref{eq:Eext}. Thus, by \cref{gen-gen}, a general semistable $E$ appears in some exact sequence \eqref{eq:Eext}.

For brevity, we put $W:= S^2 X = H^0(\PP(X), \OO(2))^\vee$. 
We parameterize injective maps $\OO(-2)^{m} \rightarrow \OO^{n}$, up to a change of basis in the source, by $\Gr := \Gr(m, \CC^n \otimes W^\vee)$, the Grassmannian of $m$ dimensional subspaces of $\CC^n \otimes W^\vee$. Given a dimension $m$ subspace $B \subset \CC^n \otimes W^\vee$, we get an injective map $\OO(-2) \otimes B \ra \OO^n$ with cokernel $F_B$. The space of extensions $E$ of $F_B$ by $T(-1)$ form a bundle over this Grassmannian. We next identify this bundle. 

Let 
\[
0 \ra \SS \ra \CC^n \otimes W^\vee 
\]
be the   universal subspace on the Grassmannian. The universal version of the map $\OO(-2) \otimes B \ra \OO^n$ is a map

\[
\OO_{\PP(X)}(-2) \boxtimes \SS  \rightarrow  \CC^n \otimes \OO
\]
on $\PP^2 \times \Gr$. Let $\FF$ be the cokernel. Its derived dual $\FF^\vee$ is given by the complex
\begin{equation}
\FF^\vee \cong [\OO \otimes (\CC^n)^\vee \ra \OO_{\PP(X)}(2) \boxtimes \SS^\vee] \label{eq:Fd}
\end{equation}

Let $p$ be the projection $\PP^2 \times \Gr \ra \Gr$. The bundle we want is
\[
\GG := R^1p_*(\FF^\vee  \otimes T(-1)).
\]
By \cref{ext-jump} (which we will prove later), there is a dense open subset $\UU$ of $\Gr$ where the dimension $H^1(\PP(X), F_B^\vee \otimes T(-1))$  is constant and equal to the expected dimension $-\chi(F_B^\vee \otimes T(-1))$. Then, by a cohomology and base change theorem, $\GG$ is a vector bundle on $\UU$, and its fibers are
\[
\GG_B = H^1(\PP(X), F_B^\vee \otimes T(-1)) \cong \Ext^1(F_B, T(-1)).
\]
as desired. So we will use $\PP(\GG) \ra \UU$ as our moduli space, modulo the choice of basis for $\CC^n$.

We can obtain a nice description of $\GG$ as follows. Tensoring \eqref{eq:Fd} with $ T(-1)$, we get that $T(-1) \otimes \FF^\vee$ is quasi-isomorphic to
\[
[T(-1) \otimes (\CC^n)^\vee \ra T(1) \boxtimes \SS^\vee]
\]
Since $T(-1)$ and $T(1)$ have no higher cohomology, we can simply apply $p_*$ to the complex to see that $Rp_*(  T(-1) \otimes \FF^\vee)$ on $\Gr$ is quasi-isomorphic to:
\begin{equation}
[(\CC^{n})^\vee \otimes H^0(T_{\PP(X)}(-1)) \rightarrow \SS^\vee \otimes H^0(T_{\PP(X)}(1))]. \label{eq:G}
\end{equation}
Then $R^1p_*(T(-1) \otimes \FF^\vee) = H^1(Rp_*(T(-1) \otimes \FF^\vee ))$ is equal to the cokernel of the above complex.

\subsection{Relationship between \texorpdfstring{$\PP(\GG)$ and $M^{ss}(e)$}{P(G) and Mss(e)}} \label{sec:relationship}
$\PP(\GG)$ carries a family of sheaves $E$ fitting into extensions \eqref{eq:Eext}. (We will describe this family in the next section.) From this family, we get a (perhaps partially defined) map $\phi: \PP(\GG) \ra M^{ss}(e)$. As discussed at the beginning of \cref{sec:the-mod}, a general semistable $E$ fits into a an extension \eqref{eq:Eext}. Furthermore, again by \cref{gen-gen}, the $F$ in that extension in general, and so lies in $\UU$. We conclude that $\phi$ is a dominant map.

We now describe the fibers of $\phi$. Let $E$ be a semistable bundle fitting into a sequence \eqref{eq:Eext}. First, one can check that for any $F$, $\Hom(T(-1),F)=0$. Hence, if we apply $\Hom(T(-1), -)$ to the sequence \eqref{eq:Eext}, we see that $\Hom(T(-1), E) \cong \Hom(T(-1),T(-1)) \cong \CC$. That is, there is a unique (up to scaling) map $T(-1) \ra E$. The isomorphism class of the cokernel $F$ does not depend on the scaling. So we see that the fiber of $\phi$ over $E$ is the set of all points $B$ in the Grassmannian $\Gr$  so that the corresponding $F_B$ is isomorphic to $F$. The following lemma shows that this set is an $SL(n)$-orbit, where the action of $SL(n)$ corresponds to the choice of basis for $\CC^n$. 

\begin{lem}
	Let $B, B' \subset \CC^n \otimes W^\vee$ be dimension $m$ subspaces, with  corresponding sequences: 
	\[
	\xymatrix{
		0 \ar[r] &B \otimes \OO(-2) \ar[r] &\CC^n \otimes \OO \ar[r]^f &F_B \ar[r] &0 \\
		0 \ar[r] &B' \otimes \OO(-2) \ar[r] &\CC^n \otimes \OO \ar[r]^{f'} &F_{B'} \ar[r] &0
	}
	\]
	Then the $F_B$ and $F_{B'}$ are isomorphic if and only if $B' = g(B)$ for some $g \in SL(n)$.
\end{lem}
\begin{proof} 
	Let $h$ be an isomorphism between $F_B$ and $F_{B'}$. Apply $\Hom(\CC^n \otimes \OO, -)$ to the sequence for $B'$. We have  $\Ext^i(\CC^n \otimes \OO, B'\otimes \OO(-2)) = 0$ for $i=0,1$, hence $\Hom(\CC^n\otimes \OO, \CC^n\otimes \OO) \cong \Hom(\CC^n\otimes \OO, F_{B'})$. Hence the map $h \circ f:\CC^n \otimes \OO \ra F_{B'}$ must be equal to $f' \circ \phi$ for a unique $\phi \in \Hom(\CC^n\otimes \OO, \CC^n\otimes \OO)$. Furthermore, we claim that $\phi$ is an isomorphism. Indeed, by the snake lemma, the kernel of $\phi$ is isomorphic to the kernel of the induced map $\psi:B \otimes \OO(-2) \ra B' \otimes \OO(-2)$. But the kernel of $\phi$ must be of the form $\OO^a$ for some $a$, while the kernel of $\psi$ must be of the form $\OO(-2)^b$ for some $b$. It follows that both $\phi$ and $\psi$ are isomorphisms. Hence we can view $\phi$ as an element of $GL(n)$, and $\phi(B) = B'$. The condition $\phi(B) = B'$ doesn't depend on a scaling factor, so we may find an appropriate $g \in SL(n)$.
	
	The other direction is straightforward.
\end{proof}
We have now seen that $\phi$ is a dominant rational map whose fibers are precisely $SL(n)$-orbits.  So we say that $\PP(\GG)/SL(n)$ is birational to $M^{ss}(e)$.
 
Instead of working directly on the quotient $\PP(\GG)/SL(n)$, we will considering equivariant bundles on $\PP(\GG)$ with their $SL(n)$-invariant sections.

\subsection{Identifying the theta bundle} \label{sec:id-theta-kron}
Label various projections as in the diagram.
\[
\xymatrix{ \PP(\GG) \ar[d]^-\pi & \ar[l]_-p  \PP(\GG) \times \PP(X) \ar[r]^-q \ar[d]^-{\pi'} & \PP(X) \ar[d]^= \\
	\Gr & \ar[l]_-{p'} \Gr \times \PP(X) \ar[r]^-{q'} & \PP(X)
}
\]

The Picard group of $\PP(\GG)$ is generated by $\OO_{\PP(\GG)}(1)$ and the pullback of the Pl\"{u}cker line bundle $p^*\OO_{\Gr}(1)$. Instead of the Pl\"{u}cker line bundle, it will be more convenient to use the bundle $p^*\det \SS^\vee$. These bundles are isomorphic, but have different actions:
\[
 \OO_{\Gr}(1) \cong \det \SS^\vee \otimes \det(\CC^n \otimes W^\vee).
\]

 As mentioned in the introduction, the theta bundle is
\[
\det^\vee(Rp_*(\EE \otimes q^*I_2  )),
\]
where $\EE$ is the family on $\PP(\GG)$ and $I_2$ is an ideal sheaf of 2 points. We first need to describe $\EE$.

In general, given sheaves $A$ and $B$ on a variety $Y$, the universal extension $\EE$ on $\PP \Ext^1(A,B) \times Y$ is an extension
\[
0 \ra \pi^*B \rightarrow \EE \rightarrow\pi^*A(-1) \ra 0  .
\]
(See, for example \cite{huybrechts_geometry_2010}, Example 2.1.12, although we have twisted our family differently.)

Applying a relative version of this to our case, we have the universal family $\EE$ of $E$'s on $\PP(\GG) \times \PP(X)$ fitting into a sequence
\[
0 \rightarrow q^*T(-1) \rightarrow \EE \rightarrow (\pi')^*\FF \otimes p^*\OO_{\PP(\GG)}(-1) \rightarrow 0.
\]

We tensor with $q^*I_2$, and push down to obtain a triangle on $\PP(\GG)$:
\[
\dots \rightarrow R\Gamma_*(T(-1) \otimes I_2) \otimes \OO \rightarrow Rp_*( \EE \otimes q^*I_2) \rightarrow Rp_*(q^*I_2 \otimes (\pi')^*\FF) \otimes\OO_{\PP(\GG)}(-1) \ra \dots
\]

Since the left object is a complex of trivial bundles, it has trivial determinant.  Hence we conclude 
\begin{align*}
	\det^\vee(Rp_*(\EE \otimes q^*I_2)) &= \det^\vee\left(Rp_*( q^*I_2 \otimes (\pi')^*\FF) \otimes \OO_{\PP(\GG)}(-1)\right)  \\
	&= \det^\vee\left(Rp_*[p^*\pi^*\SS \otimes q^*I_2(-2) \rightarrow \CC^n \otimes q^*I_2]\right)\otimes \OO_{\PP(\GG)}(-1)\\
	&=  \OO_{\PP(\GG)}(n\chi(I_2) - m\chi(I_2(-2)) \otimes \pi^*(\det\SS)^{\chi(I_2(-2))} \\
	&= \OO_{\PP(\GG)}(1) \otimes \pi^*(\det \SS^\vee)^2
\end{align*}

\subsection{Computing sections}
Unfortunately, we will only be able to directly compute a subspace of the sections that are in the kernel. We conjecture that these are all of them.

We calculated the theta bundle to be $\OO_{\PP(\GG)}(1) \otimes \pi^*(\det \SS^\vee)^2$ in \cref{sec:id-theta-kron}. We have
\begin{align*}
	 H^0(\PP(\GG), \OO_{\PP(\GG)}(1) \otimes \pi^*(\det \SS^\vee)^2) &= H^0(\UU, \pi_*(\OO_{\PP(\GG)}(1) \otimes \pi^*(\det \SS^\vee)^2)) \\
	 &= H^0(\Gr, \GG^\vee \otimes (\det \SS^\vee)^2 )
\end{align*}
The first equality is a standard fact about sections and pushforwards, and the second follows from the projection formula, the formula $\pi_*(\OO_{\PP(\GG)}(1)) = \GG^\vee$, and the fact that the complement of $\UU$ in $\Gr$ has codimension at least two (which we will prove in Lemma \ref{ext-jump}).

We saw from \eqref{eq:G} that $\GG$ fits into a sequence
\[
(\CC^{n})^\vee \otimes H^0(T_{\PP(X)}(-1)) \rightarrow \SS^\vee \otimes H^0(T_{\PP(X)}(1)) \ra  \GG \ra 0
\] 
(exact only on the right). Hence we see that  $\GG^\vee \otimes (\det \SS^\vee)^2$ is the kernel of 
\[
 \SS \otimes (\det \SS^\vee)^2 \otimes  H^0(T_{\PP(X)}(1))^\vee \rightarrow (\det \SS^\vee)^2 \otimes \CC^{n} \otimes H^0(T_{\PP(X)}(-1))^\vee
\]
One can check $\SS \otimes (\det \SS^\vee)^2 =  S^{2^{m-1}, 1} \SS^\vee$. We can now apply \cref{prop:grass-sec}, and we see that $H^0(\Gr,\GG^\vee \otimes \OO_{\Gr}(2))$ is the kernel of
\begin{equation}
 S^{2^{m-1},1}( (\CC^{n})^\vee \otimes W) \otimes H^0(\PP^2,T(1))^\vee \rightarrow S^{2^m}((\CC^{n})^\vee \otimes W) \otimes \CC^{n} \otimes H^0(\PP^2, T(-1))^\vee   \label{eq:sec-are-kernel}
\end{equation}

Next, we want to to take the $SL(n)$ invariant parts.

\begin{lem} \label{sln-inv}
	\mbox{ }
	\begin{enumerate}[label=(\alph*)]
	\item 
	\[
	S^{2^{m-1},1}( (\CC^{n})^\vee \otimes W)^{SL(n)} = S^{m,m-1}(W)
	\]	
	\item
	\[
	(S^{2^m}( (\CC^{n})^\vee \otimes W) \otimes \CC^n)^{SL(n)} =S^{m+1,m-1}(W) \oplus S^{m,m-1,1}(W)
	\]
	\end{enumerate}
\end{lem}

\begin{proof} 
We need to decompose the left hand sides as $GL(\CC^n) \times GL(W)$ modules, and then look at the parts where the $\CC^n$ part is just a power of the determinant.

In general, $S^\lambda(\CC^n \otimes W)$ can be decomposed as a $GL(\CC^n) \times GL(W)$ module:
\[
S^\lambda(\CC^n \otimes W) = \bigoplus_{\mu,\nu} \left(S^\mu(\CC^n) \otimes S^\nu(W)\right)^{\oplus g_{\mu,\nu,\lambda}}
\]
The numbers $g_{\mu,\nu,\lambda}$ are known as Kronecker coefficients, and it remains an important open problem to find a formula or combinatorial interpretation for them. 

However, one can compute them as follows using Schur functions (see \cref{sec:schur-functions}).
Let $x=(x_1, \dots, x_n)$ and $w=(w_1, \dots, w_6)$ be variables corresponding to $\CC^n$ and $W$, respectively. Write $xw = (x_1w_1, x_1w_2, \dots, x_nw_1, \dots, x_nw_6)$. Then, expand $s_\lambda(xw)$, and  collect the coefficients of monomials in the $x$-variables. Each of these coefficients is a symmetric function in $w$-variables, so it can be expressed as a linear combination of Schur functions. Also, by symmetry, $x$-monomials that differ by a permutation of the variables have the same coefficient, so we have 
\begin{align}
s_\lambda(xw) &= \sum_\alpha m_\alpha(x) \left(\sum_\nu d^{\alpha,\nu}_\lambda s_\nu(w) \right) \label{eq:mn}\\
	 &= \sum_\nu \left( \sum_\alpha d^{\alpha,\nu}_\lambda m_\alpha(x)\right) s_\nu(w)
\end{align}
Here $m_\alpha$ is the monomial symmetric function.
Now, each sum $\sum_\alpha d^{\alpha,\nu}_\lambda m_\alpha(x)$ is a symmetric polynomial in $x$, so it may be written as a sum of Schur functors, which gives the Kronecker coefficients:
\begin{align*}
	s_\lambda(xw) &= \sum_\nu \left( \sum_\mu g_{\mu,\nu,\lambda} s_\mu(x)\right) s_\nu(w) \\
	&= \sum_{\nu,\mu} g_{\mu,\nu,\lambda} s_\mu(x)s_\nu(w).
\end{align*}

For (a), let $\lambda = \lp 2^{m-1},1\rp $. We need to show that $g_{1^n, \nu, \lambda} = 1$ for $\nu = \lp m,m-1\rp $ and $0$ for all other $\nu$.

By $\mu'$ we denote the conjugate  (or transpose) of $\mu$. A nice property of the Kronecker coefficients is that $g_{\mu,\nu,\lambda} = g_{\mu',\nu',\lambda}$ \cite{pak_complexity_2017}. So we will instead check that $g_{n, \nu, \lambda} = 1$ for $\nu = \lambda$ and $0$ for all other $\nu$. (Here $n$ is the length one partition $\lp n\rp $.) The advantage is that it is easier to extract a highest weight than a lowest weight.

We now compute the coefficient of $m_n(x)$ in \eqref{eq:mn}. Consider a semistandard Young tableau of shape $\lambda$ filled with entries from $xw$, so that every box has an $x_1$ in it. It is easy to see that the sum of the content over all such tableaux is $x_1^n s_\lambda(w)$. It follows that 
\begin{align*}
	s_\lambda(xw) &= m_n(x)s_\lambda(w) + \lw(x) \\
	&= s_n(x)s_\lambda(w) + \lw(x)
\end{align*}
where $\lw(x)$ indicate terms of lower weight in $x$.

For (b), it is an exercise with the LR rule to show that  $S^{2,1^{n-1}}(\CC^n) \otimes (\CC^n)^\vee = S^{1^n}(\CC^n)$, and $S^\lambda(\CC^n) \otimes (\CC^n)^\vee$ does not contain a summand $S^{1^n}(\CC^n)$ for any other $\lambda$.  From this it follows that (b) is equivalent to the calculation of $g_{\lp 2,1^{n-1}\rp ,\lp m+1,m-1\rp,2^m } = 1$, $g_{\lp 2,1^{n-1}\rp ,\lp m,m-1,1\rp,2^m } =1$, and $g_{\lp 2,1^{n-1}\rp , \nu,2^m} = 0$ for any other $\nu$. Again, we use the conjugate property of the Kronecker coefficients to see that this is equivalent to $g_{\lp n,1\rp ,\lp 2^{m-1},1,1\rp , 2^m} = 1$, $g_{\lp n,1\rp , \lp 3,2^{m-2},1\rp , 2^m} = 1$, and $g_{\lp n,1\rp , \nu, 2^m} = 0$ for any other $\nu$.



The coefficient of $m_{n+1}(x)$ is $s_{2^m}(w)$, by the same argument as in part (a).  To find the coefficient of $m_{n,1}(x)$, note that any semistandard Young tableau of shape $2^m$ with $x$ part of the content equal to $x_1^nx_2$ must have the $x_2$ in the lower right corner. To fill in the $w$'s, notice that we can pick any semistandard Young tableau filled with $w$'s of shape $\lp 2^{m-1},1\rp $ and insert them in the boxes with $x_1$ in them, and put any $w$ in the box with $x_2$. It follows that the sum of the content of such tableaux is $s_{2^{m-1},1}(w)s_1(w)$. So we have:
\begin{align*}
	s_{2^m}(xw) = &m_{n+1}(x)s_{2^m}(w) + m_{n,1}(x)s_{2^{m-1},1}(w)s_1(w) + \lw(x) \\
				= &m_{n+1}(x)s_{2^m}(w) + m_{n,1}(x)\left[s_{2^m}(w) + s_{2^{m-1},1,1}(w) + s_{3,2^{m-2},1}(w)\right] + \lw(x)\\
				= &s_{2^m}(w)\left[m_{n+1}(x) + m_{n,1}(x) \right] + s_{2^{m-1},1,1}(w)m_{n,1}(x)  + \\ &s_{3,2^{m-2},1}(w)m_{n,1}(x)   + \lw(x) \\
				= &s_{2^m}(w)\left[s_{n+1}(x) + 0s_{n,1}(x) \right] + s_{2^{m-1},1,1}(w)s_{n,1}(x) + \\ &s_{3,2^{m-2},1}(w)s_{n,1}(x)  + \lw(x)
\end{align*}
where the $\lw(x)$ represents terms of weight lower than $\lp n,1\rp $ in $x$. The desired result follows.
\end{proof}

It is an easy computation with the Euler sequence to see that $H^0(\PP(X),T(-1))^\vee = X^\vee$ and $H^0(\PP(X),T(1))^\vee= X^\vee \otimes W - X$.

We formally subtract the right hand side of \eqref{eq:sec-are-kernel} from the left, obtaining:
\[
S^{2^{m-1},1}(\CC^{n} \otimes W) \otimes H^0(\PP^2,T(1))^\vee - S^{2^m}(\CC^{n} \otimes W) \otimes (\CC^{n})^\vee \otimes H^0(\PP^2, T(-1))^\vee
\] 
Now, taking $SL(n)$invariants and using \cref{sln-inv}, we obtain: 
\begin{align}
	=& S^{m,m-1}(W) \cdot (X^\vee \cdot W - X) - (S^{m+1,m-1}(W) + S^{m,m-1,1}(W))\cdot X^\vee \nonumber \\
	=& \left(S^{m+1,m-1}(W)  +  S^{m,m}(W)  + S^{m,m-1,1}(W)\right) \cdot X^\vee - S^{m,m-1}(W)\cdot X - \nonumber \\
	&\;\;\;\;\;\;\;\left(S^{m+1,m-1}(W) + S^{m,m-1,1}(W)\right) \cdot X^\vee  \nonumber \\
	=& S^{m,m}(W) \cdot X^\vee - S^{m,m-1}(W)\cdot X \label{eq:sec-calc}
\end{align}
If we express the last line above as a sum of Schur functors of $X$, then it follows that the summands with positive coefficients must be contained in the kernel of the map \eqref{eq:sec-are-kernel}, which kernel is the sections of the theta bundle.

In order to eliminate the dual, we recall that $X^\vee \otimes \det X = S^{1,1}(X)$ and multiply  \eqref{eq:sec-calc} by $\operatorname{det}(X)=S^{1,1,1}(X)$ and obtain
\[
S^{m,m}(W) \cdot S^{1,1}(X) - S^{m,m-1}(W) \cdot S^{2,1,1}(X)
\]

It follows easily from the Littlewood-Richardson rule that if $\lambda$ has $k$ non-zero parts and $\nu$ has fewer than $k$ non-zero parts, then $c^\nu_{\lambda, \mu} = 0$ (for any $\mu$). Thus we see that there is no summand of the form $S^{a,b}(X)$ in $S^{m,m-1}(W) \cdot S^{2,1,1}(X)$.

Hence, in order to prove   of \cref{thm-sec}\ref{thm-sec-kron}, we only need to check that $S^{m,m}(W) \cdot S^{1,1}(X)$ contains the claimed summands. It is an exercise with the Littlewood-Richardson rule to see that $S^{a,b}(X) \cdot S^{1,1}(X)$ contains a summand  $S^{a+1,b+1}(X)$. Thus, \cref{thm-sec}\ref{thm-sec-kron} follows from the next lemma.

\begin{lem} \label{summand}
For $0 \le k \le \frac m2$, there is a summand of $S^{3m-2k,m+2k}(X)$ in $S^{m,m}(W)=S^{m,m}(S^2(X))$.
\end{lem}

\begin{remark}
The proof of \cref{summand} will involve some work with the Littlewood-Richardson rule. From now on, we abbreviate ``Littlewood-Richardson" as LR. We the reader to any standard reference, for example \cite{leeuwen_littlewood_1999}, for definitions. 

Consider a skew diagram $\lambda / \mu$, where $\lambda$ (and thus also $\mu$) has two or fewer rows. 
The coefficient $c^\lambda_{\mu,\nu}$ will of course be zero unless $\ell(\nu) \le 2$ and $|\nu| + |\mu| = |\lambda|$. Given such a partition $\nu$, there is at most one LR tableau with shape $\lambda / \mu$: put 
$\nu_2$ 2's on the right of the second row, and 1's elsewhere. This will produce a valid LR tableau if and only if 
\begin{enumerate}[label=(\alph*)]
	\item The number of 2's is less than or equal to the number one 1's on the first row, or in other words: $\nu_2 \le \lambda_1 -\mu_1$. \label{few2}
	\item  We don't get two 1's in a column, in other words $\mu_1+\nu_2 \ge \lambda_2$. \label{no1}
\end{enumerate}

\end{remark}

\begin{proof}[Proof of \cref{summand}]
In Theorem 3.1 of \cite{boeck_decompositions_2014}, de Boeck and Paget give a formula to compute the multiplicity of $S^\lambda(X)$ in $S^{a,b}(S^2 X)$. Applied to our situation, the formula says that for any $\lambda$, the coefficient of $S^\lambda(X)$ is 
\begin{equation}
 \sum_{\alpha, \beta} c_{2\alpha,2\beta}^\lambda - \sum_{\gamma,\delta} c_{2\gamma,2\delta}^\lambda. \label{eq:bp}
\end{equation}
where the first the sum is over $\alpha, \beta$ partitions of $m$, and the second sum is over all $\gamma$ partitions of $m+1$ and $\delta$ partitions of $m-1$. We will apply the formula when $\lambda = \lp 3m-2k,m + 2k\rp $.

 Given a partition $\alpha=\lp \alpha_1,\alpha_2\rp $, let $\alpha^+ = \lp \alpha_1+1,\alpha_2\rp $, and $\alpha^-=\lp \alpha_1-1,\alpha_2\rp $. Of course $\alpha^-$ is not always defined.

There is a partial correspondence between the set of pairs of  partitions $(\alpha, \beta)$ and the set of pairs $(\gamma, \delta)$, given by $(\alpha, \beta) \mapsto (\alpha^+, \beta^-)$. In many cases, we will see that $c_{2\alpha,2\beta}^\lambda = c_{2\alpha^+,2\beta^-}^\lambda$, giving some cancellation in \eqref{eq:bp}. We now quantify the failure of this cancellation. For simplicity of notation, we do the argument first for $m$ even.


Suppose that $c_{2\alpha,2\beta}^\lambda =1$, so there exists an LR tableau of skew shape $\lambda / 2\alpha$ and weight $2\beta$. We can attempt to construct a LR tableau of skew shape $\lambda / 2\alpha^+$ and weight $2\beta^-$ by simply deleting 2 boxes (if possible) from the left side of the top row. Condition \ref{no1} (no column of ones) for the new tableau is automatic, and the new tableau will fail condition \ref{few2} (not too many 2's) if and only if $2\beta_2=\lambda_1-2\alpha_1$ in the original tableau. This also covers the case where there are no boxes in the first row, that is when $2\alpha_1=\lambda_1$. (The situation $2\beta_2+1=\lambda_1-2\alpha_1$ doesn't occur due to parity.)
 For a fixed $k$, the possibilities are $\alpha = \lp m-t,t\rp $ and $\beta = \lp m-\beta_2, \frac12\lambda_1-\alpha_1\rp  = \lp \frac{m}{2} + (k-t),\frac{m}{2}-(k-t)\rp $ for $t$ from $0,1, \dots k$. (When $t=k$, this covers the cases where $\beta=\lp \frac m2, \frac m2\rp $ so that $\beta^-$ is not defined.) So we see that that there are $k+1$ summands from the first sum of \eqref{eq:bp} that are not canceled out in the second sum via this correspondence. 
 
 Now suppose that $c_{2\gamma,2\delta}^\lambda=1$, so there exists an LR tableau of skew shape $\lambda / 2\gamma$ and weight $2\delta$. We can attempt to construct an LR tableau of skew shape $\lambda / 2\gamma^-$ and weight $2\delta^+$ by adding two boxes onto the the left side of the top row and putting 1's in them. Condition \ref{few2} (not too many 2's) is automatic, but condition \ref{no1} (no column of ones) will fail if and only if in the original tableau $2\gamma_1 + 2\delta_2 = \lambda_2$. For a fixed $k$, we get $\delta=\lp m-1-t,t\rp $, $\gamma=\lp \frac m2 + (k-t), \frac m2 + 1 - (k - t)\rp $ for $t=0,1, \dots, (k-1)$. So we see that there are $k$ summands from the second sum in \eqref{eq:bp} that are not canceled out in the first sum.
 
 Putting the previous two paragraphs together, we see that the coefficient of $S^\lambda(X)$ is 1, as desired.
 
 The case when $m$ is odd can be handled in the same way. In this case, the things that don't cancel from the left sum in \eqref{eq:bp} satisfy $2\beta_2 + 1 =\lambda_1-2\alpha_1$, giving $\alpha=\lp m-t,t\rp $ and $\beta = \lp \frac{m+1}{2} + (k-t),\frac{m-1}{2} - (k-t)\rp $ for $t=0, \dots, k$. The terms that don't cancel from the right sum in \eqref{eq:bp} are those that satisfy $2\gamma_1 + 2\delta_2 = \lambda_2+1$, giving $\delta=\lp m-1-t,t\rp $ and $\gamma=\lp \frac{m+1}{2} +(k-t), \frac{m-1}{2} + 1 - (k-t)\rp $ for $t=0,\dots,(k-1)$.
\end{proof}


\section{Quot schemes and injectivity} \label{sec:quot}

\subsection{Finite Quot schemes} The finite Quot scheme idea comes originally from \cite{marian_level-rank_2007}, where is was used to prove strange duality for curves. Bertram, Goller, and the author applied this method to surfaces in \cite{bertram_potiers_2016}.  The idea is as follows.   

Let $e$ and $f$ be candidates for strange duality, as in the introduction. For $(a,b,c) \in H^*(\PP^2,\QQ)$, define $(a,b,c)^\vee = (a,-b,c)$ so that $\ch(E^\vee) = \ch(E)^\vee$. Let $v = e^\vee + f$ and suppose 
$V$ is a coherent sheaf with $\ch(V) = v$. 
Then each element of the Grothendieck quot scheme $\Quot(V,f)$ of coherent sheaf quotients of $V$ of Chern character $f$ corresponds to an exact sequence
\begin{equation}
0 \rightarrow \hat E \rightarrow V \rightarrow F \rightarrow 0 \label{eq:evf}
\end{equation}
where $\ch F = f$, $\ch \hat E = e^\vee$, and $\chi(\hat E^\vee \otimes F) = \chi(e \cdot f) = 0$ is the expected dimension of $\Quot(V,f)$. 

Now suppose a sufficiently general $V$ may be chosen so that $\Quot(V,f)$ is finite and reduced, and so that each quotient
\begin{equation}
0 \rightarrow \hat E_i \rightarrow V \rightarrow F_i \rightarrow 0 \label{eq:evfi}
\end{equation}
has the property that $F_i$ is torsion free and $\hat E_i^\vee$ and $F_i$ correspond to points in the respective moduli spaces. ($\hat E_i$ is locally free since it is the kernel of a surjective map from a locally free sheaf to a torsion free sheaf.) The finite and reduced condition implies that $\Hom(\hat E_i,F_i) = 0$, which says that $\Theta_{F_i}([\hat E_i^\vee])$ does not vanish. Furthermore, by composing maps from \eqref{eq:evfi}  we see that $\hom(\hat E_i, F_j) \neq 0$ for $i \neq j$ (use the torsion free condition here).  This shows that $\Theta_{F_i}([\hat E_j^\vee])$ does vanish. Hence the ``matrix'' $\Theta_{F_i}([\hat E_j^\vee])$ is diagonal with non-zero entries on the diagonal, and we conclude that the sections $\Theta_{F_i}$ are linearly independent. As we observed below \eqref{eq:sd} in the introduction, a section $\Theta_{F}$ is in the image of $\operatorname{SD}_{e,f}$ for any $F$. So the rank of $SD_{e,f}$ is at least the number of points in the Quot scheme. In particular, if the length of the Quot scheme is equal to the dimension of $H^0\big(M_S(e^{\vee}),{\mathcal O}(\Theta_f)\big)$, we can conclude that $\operatorname{SD}_{e,f}$ is surjective, and if the length of the Quot scheme is equal to the dimension of $H^0\big(M_S(f),{\mathcal O}(\Theta_{e^{\vee}})\big)^*$, we can conclude that $\operatorname{SD}_{e,f}$ is injective.

There are two main difficulties in using the Quot scheme method. First, one must be able to check that both the kernels and quotients appearing in sequences \eqref{eq:evfi} are points of your moduli spaces. Addressing the question of quotients, we have the even stronger,
\begin{thm}[Theorem B, \cite{bertram_potiers_2016}] \label{thm:B}
	Suppose $k\ge 1$, $r \ge 2$, and $d\gg 0$. Let $e = (r, d, s)$ and $f = 1 - k \cdot h^2$, where $s$ is determined by the condition $\chi(e \cdot f) = 0$. Let $v = e^\vee + f$.
	
	 Let $V$ be a general stable vector bundle on $\PP^2$ with $\ch(V)=v$ as above. Then the quot scheme $\Quot(V,(1,0,-k))$ is finite and reduced, and each quotient is an idea sheaf of a reduced subscheme.
\end{thm}

Secondly, we must be able to compare the cardinality of the Quot scheme to the dimensions of the spaces of sections. In \cite{bertram_potiers_2016} an expected size of the Quot scheme was computed using multiple point formulas from \cite{marangell_general_2010} and \cite{berczi_thom_2012}. It is difficult to address the question of when this count is correct, since the multiple point formulas require a genericity  that cannot be verified for algebraic maps. In addition, the computation of these numbers is slow, and the formulas are limited to $k \le 7$. 

In \cite{johnson_universal_2017}, the present author suggested a way to interpret the expected size of the Quot scheme as a Chern number of a certain tautological vector bundle on $(\PP^2)^{[k]}$, and conjectured that these numbers matched the Euler characteristics of the appropriate theta bundles on $(\PP^2)^{[k]}$. In fact, this conjecture is relevant to all surfaces. This conjecture has now been verified computationally up to $k=11$. (We only need $k=2$ for this paper.) 

In \cite{goller_rank1_2019}, Goller and Lin proved that, under suitable hypothesis on $V$, this Chern number does in fact compute the number of points on the Quot scheme. Combining this with \cref{thm:B} and the computational verification of the conjecture from \cite{johnson_universal_2017}, they obtained:
\begin{thm}[Theorem 1.2, \cite{goller_rank1_2019}] \label{thm:goller-lin}
	Assume $1 \le k \le 11$, $r \ge 2$, and $d\gg 0$. Let $V$ be a general stable sheaf as in \cref{thm:B}. 
	
	Then
	\[
	\#\Quot(V,k) = h^0((\PP^2)^{[k]}, \OO(\Theta_e))
	\]
\end{thm}

In order to use this strategy to prove \cref{thm:inj} and show that our strange duality map is injective for the cases in this paper, there are two tasks remaining, which we address in the following two sections.

\subsection{The condition \texorpdfstring{$d \gg 0$}{d>>0}.} First, we must verify that the first Chern class $d$ of $V$ is sufficiently large to be able to apply  \cref{thm:B}. \cref{thm:goller-lin} needs this hypothesis only to apply \cref{thm:B}.

We first remark on a discrepancy in notation from \cite{bertram_potiers_2016}: the roles of $E$ and $\hat E$ have been reversed, and what we call $e$ here corresponds to $e^\vee$ there.

The proof of  \cref{thm:B} in \cite{bertram_potiers_2016} uses what is there called a \ref{dag}-resolution of $V$ in an essential way. We can show that a general $V$ with our invariants has a \ref{dag}-resolution only when $m \le 5$, hence the restriction on \cref{thm:inj}.
\begin{lem}
	Let $v = e^\vee + f$, where $e=(m+1, 2m+1, -2m-\frac12)$, $f=(1,0,-2)$, and $m \le 5$.
	Then the general semistable $V$ with $\ch(V) = v$ has a general \ref{dag}-resolution:
	\begin{equation}
	0 \rightarrow V \rightarrow \OO(1)^3 \oplus \OO^{2m+1} \rightarrow \OO(2)^{m+2} \rightarrow 0 \tag{$\dagger$} \label{dag}
	\end{equation}
\end{lem}
\begin{proof}
	A general $V$ is locally free, and $V^\vee$ is a general semistable sheaf of Chern character $v^\vee$ (see \cite{huybrechts_geometry_2010}, Lemma 9.2.1). Apply  \cref{beil-res} to $v^\vee$. There are only 5 cases, so checking the numerical conditions is routine and we see that $V^\vee$ has a ``dual \ref{dag}-resolution"
	\[
	0 \ra \OO(-2)^{m+2} \ra \OO(-1)^3 \oplus \OO^{2m+1} \ra  V^\vee \ra 0.
	\] 
	As in \cref{gen-gen}, we can also conclude that this resolution is general. Then its dual gives \eqref{dag} as desired.
\end{proof}

In Lemma 5.7 of \cite{bertram_potiers_2016}, the dimension count requires a numerical condition that is satisfied for $d \gg 0$. Plugging in our numbers, the condition is equivalent to $1(1+ (m+1) + 1) < 3 + (2m+1)$, so we are ok.

In Section 5.3 \cite{bertram_potiers_2016}, $e$ and $f$ are verified to be candidates for strange duality if $d \gg 0$. The condition is used in Lemma 5.8 of \cite{bertram_potiers_2016} which says that a general semistable $E$ and a general ideal sheaf $I_Z$ of two points satisfy $h^0(E \otimes I_Z) = 0$. Examining the proof, we see the conclusion holds when $d > j\operatorname{rank}(E)$, where $\chi(\OO_{\PP^2}(j)) \ge k$ (where $f=(1,0,-k)$). In our case, $k=2$, so we can take $j=1$, and the requirement is $2m+1 > m+1$. 

Section 5.4 of of \cite{bertram_potiers_2016} constructs a necessary $V$ with an exact sequence 
\[
 0 \ra \hat E \ra V \ra I_2 \ra 0
\]
so that $h^0(\hat E \otimes I_2) = 0$. Our \cref{subs-of-V-gen} shows that a general stable $V$ has general kernels and quotients, so the desired result follows from the previous paragraph. 

These are all the requirements on $d \gg 0$.

\subsection{Kernels are points of  \texorpdfstring{$\PP(\GG)$}{P(G)}}  \label{sec:res}

Secondly, we need to worry about whether the (duals of the) kernels appearing in the Quot scheme are actually in our moduli space $\PP(\GG)$.  This will be addressed by the following theorem.

\begin{thm} \label{kernel-in-moduli}
	Suppose $V$ has a general \ref{dag}-resolution. 
	If we have any short exact sequence
	\[
	0 \ra \hat E \ra V \ra I_2 \ra 0,
	\]  
	where $I_2$ is an ideal sheaf of 2 points, then $\hat E$ is locally free, and its dual $E:=\hat E^\vee$ fits into a short exact sequence \eqref{eq:Eext}, and the $F$ appearing there satisfies $0=\Ext^0(F,T(-1)) = \Ext^2(F,T(-1))$.
\end{thm}
Combined with \cref{thm:B} and \ref{thm:goller-lin} and the discussion of the finite Quot scheme method above, this proves \cref{thm:inj}. We give the proof of \cref{kernel-in-moduli} at the end of this section after proving the necessary lemmas.

\begin{lem}
	Suppose $V$ has a \ref{dag}-resolution. Then any surjective map $V \rightarrow I_2$, where $I_2$ is an ideal sheaf of two distinct points, extends to a commutative diagram, where all the vertical arrows are surjective.
	\begin{equation}
	\xymatrix{
		0 \ar[r] & V \ar[r] \ar[d] & \OO(1)^3 \oplus \OO^{2m+1} \ar[r]^-\psi \ar@{-->}[d] & \OO(2)^{m+2} \ar[r] \ar@{-->}[d] & 0 \\ 0 \ar[r] & I_2 \ar[r]& \OO \ar[r] & \OO_2 \ar[r] & 0
	} \label{eq:quot}
	\end{equation}
\end{lem}
\begin{proof}
	From the long exact sequence obtained by applying $\Hom(-,\OO)$ to the resolution of $V$, we get $\Hom(\OO(1)^3 \oplus \OO^{2m+1}, \OO) \cong \Hom(V, \OO)$. Hence the composition $V \rightarrow I_2 \rightarrow \OO$ lifts to a map $\OO(1)^3 \oplus \OO^{2m+1} \rightarrow \OO$. The map $\OO(2)^{m+2} \rightarrow \OO_2$ is induced by the first two.
	
	There are no non-zero maps $\OO(1)^3 \ra \OO$, so it follows that $\OO^{2m+1} \rightarrow \OO$ must be non-zero, which implies that it is surjective. Then the rightmost vertical map is then surjective by the snake lemma. 
\end{proof}

Let 
\begin{equation}
	0 \ra E \ra \OO(1)^3 \oplus \OO^{2m} \ra K \ra 0 \label{eq:seq-ker}
\end{equation}
 be the sequence of kernels in the diagram \eqref{eq:quot}.

\begin{lem}
	Assume $H^1(K(-2))=0$. Then $K$ has a resolution
	\begin{equation}
	0 \ra \OO^2 \rightarrow \OO(1)^4 \oplus \OO(2)^m \rightarrow K \rightarrow 0 \label{eq:Kres}
	\end{equation}
\end{lem}
\begin{proof}
	This will follow from applying Proposition 5.5 in \cite{bertram_potiers_2016} to $K(-2)$. To use Proposition 5.5, we need to check that $0=H^0(K(-3)) = H^1(K(-2)) = H^2(K(-3))$ and that the natural map $\Hom(\OO(-1), \OO) \otimes H^0(K(-2)) \ra H^0(K(-1))$ is injective. Twisting the defining sequence
	\[
	0 \ra K \ra \OO(2)^{m+2} \ra \OO_2 \ra 0,
	\]
	one easily computes that $0=H^0(K(-3)) = H^2(K(-3))$.   
	
	Tensor the Euler sequence $0 \ra \Omega \ra \Hom(\OO(-1),\OO) \otimes \OO(-1) \ra \OO \ra 0$ with $K(-1)$ and take sections. The desired injectivity is now equivalent to $H^0(\Omega \otimes K(-1))=0$. Since $\Omega \otimes K(-1)$ is a subsheaf of $\Omega(1)^{m+2}$, we are done.
 
 	The exponents $2$, $4$, and $m$ are determined by the Chern character of $K$. 
\end{proof}

\begin{lem}
	If $V$ has a general \ref{dag}-resolution, then it does not admit a diagram \eqref{eq:quot} with $H^1(K(-2)) >0$.
\end{lem}
\begin{proof}
	The condition $H^1(K(-2))>0$ is equivalent to the condition that  the map  $f:\OO^{m+2} \ra \OO_2$ obtained by twisting the rightmost vertical map in  \eqref{eq:quot}  is not surjective on sections.
	
	Let $A$ be a $3 \times (m+2)$ matrix of linear forms and $B$ be a $(2m+1) \times (m+2)$ matrix of quadratic forms representing the map $\psi$ in \eqref{eq:quot}. Let $p_1$, $p_2$ be distinct points in $\PP(X)$. We can identify the fibers $(\OO(1)^3 \oplus \OO^{2m+1})_{p_i} \cong \CC^3 \oplus \CC^{2m+1}$ and $(\OO(2)^{m+2})_{p_i} \cong \CC^{m+2}$ by picking lifts $v_i \in X$ of $p_i$. Then the map $\psi$ restricted to the fiber over $p_i$  is given by $A(v_i)$ and $B(v_i)$. We can represent the map  $f$ by two row vectors $c_1$, $c_2 \in \CC^{m+2}$ (which give the maps $ \CC^{m+2} \cong (\OO(2)^{m+2})_{p_i} \ra \OO_{p_i} \cong \CC$). Commutativity of the diagram \eqref{eq:quot} imposes the requirements that $c_iA(v_i)  = 0$ and $c_1B(v_1) = c_2B(v_2)$.
	
	If furthermore, we would like $f$ to fail to be surjective on sections, we must have additionally $c_1 = dc_2$ for some $d \in \CC$. Given $p_1$ and $p_2$, such $c_i$ exist if and only if there are lifts $v_i$ so that the matrix
	\begin{equation}
	\left[A(v_1)\;|\; A(v_2)\;|\;B(v_1) - B(v_2) \right] \label{eq:ABmat}
	\end{equation}
	is not full rank. Indeed, if it is not full rank, let $c_1=c_2$ be a left solution to the homogeneous equation. On the other hand, if we have $c_iA(v_i)=0$  and $c_1B(v_1) = dc_1B(v_2)$, then let $v_1'=v_1$ and $v_2' = v_2/\sqrt{d}$, so $c_1B(v_1')=c_1B(v_2')$. Then $c_1$ is a left solution.
	
	Next, we claim that the set of bad pairs $A$, $B$, that is, those such that there exists $v_1$, $v_2$ so that \eqref{eq:ABmat} is not full rank, has positive codimension in the space of all $A$ and $B$.
	
	Consider the maps
	\[
	\xymatrix{
	M(X^\vee, 3,m+2) \times M(S^2X^\vee, 2m+2,m+2) \times (X \times X \setminus \Delta) \ar[r]^-{e} \ar[d]_{p}  &M(\CC, 2m+8, m+2) \\ 	M(X^\vee, 3,m+2) \times M(S^2X^\vee, 2m+2,m+2)
}
	\]
	Here, $M(R,a,b)$ means the space of $a \times b$ matrices with entries in $R$, $\Delta$ is the diagonal in $X \times X$, and the map $e$ takes $(A, B,v_1,v_2)$ to \eqref{eq:ABmat}. Let $Z \subset M(\CC, 2m+8, m+2) $ be  the space of rank deficient matrices, which has codimension $m+7$. Then the space of bad $A$ and $B$ is precisely $p(e^{-1}(Z))$. Now, we check that the fibers of $e$ are all of the same dimension. Indeed, the fiber over a point $[\bar A_1\;|\;\bar A_2\;|\;D]$ consists of $(A,B,v_1,v_2)$ such that $v_1 \neq v_2$, the entry $A_{jk}$ is a linear form satisfying $A_{jk}(v_i) = (\bar A_i)_{jk}$, and $B_{jk}$ is a quadratic form satisfying $B_{jk}(v_1) = (\bar B_1)_{jk}$ and $B_{jk}(v_2) = (\bar B_1-D)_{jk}$ for some $\bar B_1 \in M(\CC, 2m+2,m+2)$. Since $v_1, v_2, \bar B_1$ have no restrictions (except the closed condition $v_1 \neq v_2$), and fixing the values of a linear or quadratic form at two distinct points always imposes independent conditions, we see that the fibers are all of the same dimension as claimed. 

	It follows then that the codimension of $e^{-1}(Z)$ is also $m+7$. Since $p$ has relative dimension $6$, the image $p(e^{-1}(Z))$ has codimension at least $m+1$.
\end{proof}

\begin{lem} \label{ker-res}
	If $V$ has a general \ref{dag}-resolution, then for any ideal sheaf of 2 points $I_2$ with a sequence
	\[
	0 \ra \hat E \ra V \ra I_2 \ra 0,
	\] $\hat E$  has a \ref{res-d}-resolution:
	\begin{equation}
	0 \rightarrow \hat E \rightarrow \Omega(1) \oplus \OO^{2m-1} \rightarrow \OO(2)^m \rightarrow 0. \tag{$\star^\vee$} \label{res-d}
	\end{equation}
\end{lem}
\begin{proof}
	We first claim that $\hat E$ has a resolution
	\[
	0 \ra \hat  E \ra \OO(1)^3 \oplus \OO^{2m+2} \rightarrow \OO(1)^4 \oplus \OO(2)^m \rightarrow 0.
	\]
	Apply $\Hom( \OO(1)^3 \oplus \OO^{2m}, -)$ to the resolution \eqref{eq:Kres} for $K$. Since $\Ext^1(\OO(1)^3 \oplus \OO^{2m},\OO^2) = 0$, we see that the map $\OO(1)^3 \oplus \OO^{2m} \ra K$ from \eqref{eq:seq-ker} factors  through the last map in of the $K$ resolution \eqref{eq:Kres}. That is, we get a diagram:
	\[
	\xymatrix{
	0 \ar[r] & \hat E \ar[r] \ar[d]^{a} & \OO(1)^3 \oplus \OO^{2m} \ar[r] \ar[d]^{b} & K \ar[r] \ar[d] & 0 \\
	0 \ar[r] & \OO^2 \ar[r] & \OO(1)^4 \oplus \OO(2)^m  \ar[r] &           K \ar[r]          & 0
	}
	\]
	The maps for the claimed resolution come from this diagram.
	
	Now, consider the map $b': \OO(1)^3 \ra \OO(1)^4$ coming from the map $b$ above, and write $\OO(1)^f$ for its image and $\OO(1)^{4-f}$ for its cokernel.
	
	Now consider the diagram, with the middle two columns exact.
	\[
	\xymatrix{ 
		&0 \ar[r]&\ar[r] \OO(1)^3 \ar[r]^{b' \oplus 0} \ar[d] & \OO(1)^f \oplus \OO(2)^m  \ar[d]^{c}&  \\
		0 \ar[r] & \hat E \ar[r] & \OO(1)^3 \oplus  \OO^{2m+2} \ar[r] \ar[d] & \OO(1)^4 \oplus \OO(2)^m \ar[r] \ar[d]&0 \\ && \OO^{2m+2} \ar@{-->}[r]^d & \OO(1)^{4-f}
	}
	\]
	 By the snake lemma, it follows that the bottom map $d$ is surjective. Pick a subsheaf $\OO(1) \subset \OO(1)^{4-f}$, and a a subsheaf $\OO^3 \subset \OO^{2m+2}$ so that the restriction of $d$ is surjective onto the chosen $\OO(1)$.
	
	Now consider the diagram:
	\[
	\xymatrix{ 
		0 \ar@{-->}[r] & \hat E \ar@{-->}[r] \ar@{-->}[d]^-{\sim}& \Omega(1) \oplus \OO^{2m-1} \ar@{-->}[r] \ar@{-->}[d] & \OO(2)^m \ar@{-->}[r] \ar@{-->}[d] & 0 \\
		0 \ar[r] & \hat E \ar[r] & \OO(1)^3 \oplus \OO^3 \oplus \OO^{2m-1} \ar[r] \ar[d] & \OO(1)^4 \oplus \OO(2)^m \ar[r] \ar[d]&0 \\ && \OO(1)^3 \ar[r]^-{\sim} \oplus \OO(1) & \OO(1)^4}
	\]
	where the $\OO^3$ is the subspace we picked above. Taking kernels, we obtain the dotted lines and the desired resolution.
\end{proof}

We need to check that if we pick $V$ generally, then we can ensure that the kernels $\hat E$ are general.

\begin{lem} \label{subs-of-V-gen}
	A general semistable $V$ has general $I_2$ quotients and kernels with general resolutions. 
	
	More precisely, let $U \subset \Hom(\OO^n \oplus \Omega(1), \OO(2)^m)$ be the (open) subset of surjective maps. Let $Z \subset U \times \PP(X)^{[2]}$ be a closed subset that respects isomorphism of kernels, that is, if $\ker \phi \cong \ker \phi'$, then for each $I_2\in \PP(V)^{[2]}$, either $\phi, \phi' \in Z|_{I_2}$ or $\phi, \phi' \notin Z|_{I_2}$.
	
	Then a general semistable $V$ has no short exact sequence
	\[
	 0 \ra \hat E \ra V \ra I_2 \ra 0
	\]
	with $\hat E = \ker \phi$ for some $(\phi,I_2) \in Z$.   
\end{lem}
\begin{proof}
	Consider an $\hat E$ with a resolution \eqref{res-d}.
	By applying $\Hom(-,\OO^n \oplus \Omega(1))$, we see that $\Hom(E, \OO^n \oplus \Omega(1)) = \Hom(\OO^n \oplus \Omega(1),\OO^n \oplus \Omega(1))$, which has dimension $n^2 + 3n + 1$. Furthermore, one can change the basis in the cokernel $\OO^m$, adding $m^2$. Finally, we subtract 1 for simultaneous scaling, and see that $\hat E$ actually has a $f_d := n^2 + 3n+ m^2$ dimensional space of such resolutions. 
	
	Let $T_0 \subset  U \times \PP(X)^{[n]}$ be the space of pairs $(\phi, I_2)$ so that $\ext^1(I_2, \ker \phi)$ has the expected dimension $-\chi(I_2, \ker \phi)$.  There is a projective bundle over $T_0$ whose fibers are the spaces of extensions. There is a (perhaps partially defined) map from this bundle to $M^{ss}(v)$. By the discussion in the first paragraph, the fibers of this map have dimension at least $f_d$. We calculate
	\begin{align*}
	 &\hom(\OO^n \oplus \Omega(1), \OO(2)^m)   -\chi(I_2,e) -1 - f_d + \dim \PP(X)^{[2]} \\
	= &6nm + 15m  + (6m +3)  - 1 - \left[n^2 + 3n+ m^2 \right] + 4 \\
	= &7m^2 + 13m + 8  \\
	= &\dim M^{ss}(v)
	\end{align*}
	This shows that no closed subset of $T_0$ could give rise to a general $V$.
	
	Next, we claim that the set of $V$ that occur as extensions of pairs $(\phi,I_2)$ with too many extensions is contained in a closed subset. For an $I_2 \in \PP(X)^{[2]}$ and $j > 0$, let $T_j$ be the subset  of $\phi$ where $\ext^1(I_2, \ker \phi) = - \chi(I_2, \ker \phi) + j$. Let $Y_j$ be the projective bundle over $T_j$ parameterizing these extensions. There is a rational map $Y_j \ra M^{ss}(v)$. Again, the fibers of this map have dimension at least $f_d$. Thus the lemma will be proved if we can show that $\dim Y_j - f_d + \dim \PP(X)^{[2]} < \dim M^{ss}(v)$. 
	
	Let $\phi \in T_j$ and $\hat E  = \ker \phi$. We may also assume that  $\hat E$ is a subsheaf of a semistable $V$, in particular, it has no sections. One can check from the resolution that $\Ext^1(\OO_2, \hat E)=0$, and it follows that $\Hom(I_2, \hat E)=0$, so $\ext^1(I_2,\hat E) = -\chi(I_2, \hat E) + \ext^2(I_2,\hat E)$.
	
	Consider the map $\Ext^1(I_2, \Omega(1) \oplus \OO^{n}) \ra \Ext^1(I_2, \OO(2)^m)$. Its cokernel is $\Ext^2(I_2, \hat E)$. We will investigate the codimension of the locus of $\hat E$ where $\ext^2(I_2, \hat E) \ge j$. 
	
	By the standard long exact sequence for $I_2$ and Serre duality, we see that if a locally free sheaf $A$ has $H^1(A) = H^2(A) = 0$, then there are natural isomorphisms
	\begin{align*}
		\Ext^1(I_2, A) &\cong \Ext^2(\OO_Z, A) \cong \Hom(A, \OO_Z \otimes K)^\vee \\ &\cong H^0((A^\vee \otimes K)|_Z )^\vee \cong H^0((A \otimes K^\vee)|_Z )
	\end{align*}
	Hence the map $\Ext^1(I_2, \Omega(1) \oplus \OO^{2m-1}) \ra \Ext^1(I_2, \OO(2)^m)$ induced by $\phi$ is equivalent to the map $H^0((\Omega(1) \oplus \OO^{n})|_Z) \ra H^0(\OO(2)^m|_Z)$ on fibers induced by $\phi$.  Given any such map on fibers, one can check that it can be extended to a map on sheaves $\Omega(1) \oplus \OO^{2m-1} \ra \OO(2)^m$. That is, the map 
	\[
	\Hom(\Omega(1) \oplus \OO^{n},\OO(2)^m) \ra \Hom(\Ext^1(I_2, \Omega(1) \oplus \OO^{n}), \Ext^1(I_2, \OO(2)^n))
	\]
	is surjective. One calculates 
	\[
	\ext^1(I_2, \Omega(1) \oplus \OO^{2m-1}) = 4 + 2(2m-1), \;\;\;\; \ext^1(I_2, \OO(2)^m) = 2m.\]
	 Hence the locus in $\Hom(\Ext^1(I_2, \Omega(1) \oplus \OO^{2m-1}), \Ext^1(I_2, \OO(2)^m))$ that drops rank by at least $j$  is codimension $c_j:=(2m+2+j)j$, and the preimage of that locus in $\Hom(\Omega(1) \oplus \OO^{2m-1},\OO(2))$ has the same codimension.
	
	We have 
	\begin{align*}
	&\dim Y_j - f_d + \dim \PP(X)^{[2]} \\
	\le &\hom(\OO^n \oplus \Omega(1), \OO(2)^m) - c_j   -\chi(I_2,e) +j -1 - f_d + \dim \PP(X)^{[2]} \\
	= &6nm + 15m - c_j  + (6m +3) + j - 1 - \left[n^2 + 3n+ m^2 \right] + 4 \\
	= &7m^2 + 13m + 8 - c_j + j \\
	= &\dim M^{ss}(v) - c_j + j.
	\end{align*} 
	Since $c_j > j$ for $m \ge 1$, we are done.
\end{proof}

\begin{lem} \label{ext-jump}
For any dimension $m$ subspace $B \subset \CC^n \otimes W^\vee$, let $F_B$ be the cokernel of the map $B\otimes \OO(-2) \ra \CC^n \otimes \OO$.  Then $F_B$ satisfies $0=\Hom(F_B, T(-1)) = \Ext^2(F_B, T(-1))$  away from a subset of codimension at least 3 in $\Gr(m, \CC^n \otimes W^\vee)$.
\end{lem}
\begin{proof}
 If we pick a basis for $B$, the map $B \otimes \OO(-2) \ra \CC^n \otimes \OO$ can be described by a $n \times m$ matrix $M$ with entries in $W^\vee = H^0(\PP(X), \OO(2))$.
 
 Applying $\Hom(-,T(-1))$ to the defining sequence of $F_B$, we see that $\Ext^2(F_B, T(-1)) = 0$ for any $B$, and $\Hom(F_B, T(-1))$ is equal to the kernel of the induced map $\Hom( \OO^n, T(-1)) \ra \Hom(B \otimes \OO(-2), T(-1))$. We can also view this induced map as a map $F:H^0(T(-1))^n \ra H^0(T(1))^m$, given by taking $n$ $(-1)$-homogeneous vector fields on $X$ and multiplying by $M^T$ to obtain $m$ 1-homogeneous vector fields.

	For $m=1$, one can easily check that $F$ is injective.  We give the proof for $m \ge 2$.

	Recall that $H^0(T(-1)) = X$. Let us fix $\phi \in X^{n}$. Consider the diagram
	\[
	\xymatrix{&0 \ar[d]\\
		& H^0(\OO(2) \otimes \OO(-1)) \ar[d]  \\
		H^0(\OO(2) \otimes \CC^{n}) \ar^{\hat \phi}[r] & H^0(\OO(2) \otimes X) \ar^e[d] \\
		& H^0(\OO(2)\otimes T(-1)) \ar[d]\\
		&0}
	\]
	The vertical maps come from the twisted Euler sequence, and the horizontal map $\hat \phi$ is taking the dot product with $\phi$. We are thinking of $H^0(\OO(2) \otimes \CC^{n})$  as all possible single rows of $M^T  $. 

	First, assume that the entries of $\phi$ span $X$. This implies that $\hat \phi$ is surjective. Hence, $e \circ \hat \phi$ is also surjective, so its kernel has codimension equal to $h^0(T(1))=15$. 
	Thus, the locus of $M^T$ so that $M^T\phi = 0$ is codimension $15m$ in the space of all $M$. The space of possible $\phi$ is dimension $(2m-1)3$, so the codimension of bad $M^T$ is at least $9m+3$.

	Next, assume that the entries of $\phi$ span only a subspace $X' \subset X$ of dimension $d$ (where $d=1$ or $2$). Then the image of $\hat \phi$ has dimension $6d$. The image of $\hat \phi$ does not intersect the kernel of $e$ non-trivially, so the image of $e \circ \hat \phi$ also has dimension $6d$. Hence, the codimension of $M^T$ so that $M^T \phi = 0$ is $6dm$. But the space of possible $\phi$ is $(2m-1)d + 2$. (That is $2m-1$ choices of an entry from $X'$, and a 2 dimensional choice of $X'$.) Thus, the codimension of bad $M^T$ is at least $5d-2$.

	Finally, we remark that the action of $GL(m)$ on the space of $M^T$, corresponding to the choice of basis for $B$, does not affect whether the corresponding map $F$ is injective, so the codimensions we computed are also good on $\Gr(m, \CC^n \otimes W^\vee)$.
\end{proof}

\begin{proof}[Proof of \cref{kernel-in-moduli}]
	$\hat E$ is locally free because it is the kernel of a map from a locally free sheaf to a torsion free sheaf, and has a general \ref{res-d}-resolution by \cref{ker-res} and \cref{subs-of-V-gen}. Hence its dual $E:=\hat E^\vee$ has a general \ref{eq:Eres}-resolution. We discussed in the beginning of \cref{sec:the-mod} how such a sheaf fits into a sequence \eqref{eq:Eext}. The condition on $\Ext$ groups follows from \cref{ext-jump}.
\end{proof}

 
 \bibliographystyle{alpha}
 
 \bibliography{mybib}{}

\begin{thebibliography}{ABCH13}

\bibitem[ABCH13]{arcara_minimal_2013}
Daniele Arcara, Aaron Bertram, Izzet Coskun, and Jack Huizenga.
\newblock The minimal model program for the {Hilbert} scheme of points on and
  {Bridgeland} stability.
\newblock {\em Advances in Mathematics}, 235:580--626, March 2013.

\bibitem[BGJ16]{bertram_potiers_2016}
Aaron Bertram, Thomas Goller, and Drew Johnson.
\newblock Le {Potier}'s strange duality, quot schemes, and multiple point
  formulas for del {Pezzo} surfaces.
\newblock October 2016.
\newblock arXiv: 1610.04185.

\bibitem[BS12]{berczi_thom_2012}
Gergely Bérczi and András Szenes.
\newblock Thom polynomials of {Morin} singularities.
\newblock {\em Annals of Mathematics}, 175(2):567--629, March 2012.

\bibitem[CHW14]{coskun_effective_2014}
Izzet Coskun, Jack Huizenga, and Matthew Woolf.
\newblock The effective cone of the moduli space of sheaves on the plane.
\newblock {\em Journal of the European Mathematical Society}, page to appear,
  January 2014.
\newblock arXiv: 1401.1613.

\bibitem[Dan02]{danila_resultats_2002}
Gentiana Danila.
\newblock Résultats sur la conjecture de dualité étrange sur le plan
  projectif.
\newblock {\em Bull. Soc. Math. France}, 130(1):1--33, 2002.

\bibitem[dP14]{boeck_decompositions_2014}
Melanie {de Boeck} and Rowena {Paget}.
\newblock {Decompositions of some twisted Foulkes characters}.
\newblock Sep 2014.
\newblock arXiv:1409.0737.

\bibitem[EF09]{edidin_grassmannians_2009}
D.~Edidin and C.~A. Francisco.
\newblock Grassmannians and representations.
\newblock {\em J. Commut. Algebra}, 1(3):381--392, 09 2009.

\bibitem[GL19]{goller_rank1_2019}
Thomas {Goller} and Yinbang {Lin}.
\newblock {Rank-1 sheaves and stable pairs on del Pezzo surfaces}.
\newblock Jul 2019.
\newblock arXiv:1907.05180.

\bibitem[HL10]{huybrechts_geometry_2010}
Daniel Huybrechts and Manfred Lehn.
\newblock {\em The {Geometry} of {Moduli} {Spaces} of {Sheaves}}.
\newblock Cambridge University Press, second edition, May 2010.

\bibitem[Joh16]{johnson_two_2016}
Drew Johnson.
\newblock {\em Two enumerative problems in algebraic geometry}.
\newblock PhD thesis, University of Utah, 2016.

\bibitem[Joh17]{johnson_universal_2017}
Drew Johnson.
\newblock Universal {Series} for {Hilbert} {Schemes} and {Strange} {Duality}.
\newblock August 2017.
\newblock arXiv: 1708.05743.

\bibitem[LP05]{le_potier_dualite_2005}
J.~Le~Potier.
\newblock Dualité étrange sur le plan projectif.
\newblock {\em Unpublished}, 2005.

\bibitem[MO07]{marian_level-rank_2007}
Alina Marian and Dragos Oprea.
\newblock The level-rank duality for non-abelian theta functions.
\newblock {\em Invent. Math.}, 168(2):225--247, January 2007.

\bibitem[MO08a]{marian_sheaves_2008}
Alina Marian and Dragos Oprea.
\newblock Sheaves on abelian surfaces and strange duality.
\newblock {\em Math. Ann.}, 343(1):1--33, August 2008.

\bibitem[MO08b]{marian_tour_2008}
Alina Marian and Dragos Oprea.
\newblock A tour of theta dualities on moduli spaces of sheaves.
\newblock In {\em Curves and abelian varieties}, volume 465 of {\em Contemp.
  {Math}.}, pages 175--201. Amer. Math. Soc., Providence, RI, 2008.

\bibitem[MO14]{marian_strange_2014}
Alina Marian and Dragos Oprea.
\newblock On the strange duality conjecture for abelian surfaces.
\newblock {\em Journal of the European Mathematical Society}, 16(6):1221--1252,
  2014.

\bibitem[MOY10]{marian_generic_2010}
Alina Marian, Dragos Oprea, and Kota Yoshioka.
\newblock Generic strange duality for {K3} surfaces.
\newblock {\em arXiv:1005.0102 [math]}, May 2010.
\newblock arXiv: 1005.0102.

\bibitem[MR10]{marangell_general_2010}
R.~Marangell and R.~Rimányi.
\newblock The general quadruple point formula.
\newblock {\em Amer. J. Math.}, 132(4):867--896, August 2010.

\bibitem[PP17]{pak_complexity_2017}
Igor Pak and Greta Panova.
\newblock On the complexity of computing kronecker coefficients.
\newblock {\em computational complexity}, 26(1):1--36, Mar 2017.

\bibitem[Sca07]{scala_dualite_2007}
Luca Scala.
\newblock Dualité étrange et cohomologie du schéma de {Hilbert}.
\newblock {\em Gaz. Math.}, (112):53--65, April 2007.

\bibitem[vL99]{leeuwen_littlewood_1999}
Marc A.~A. van {Leeuwen}.
\newblock {The Littlewood-Richardson rule, and related combinatorics}.
\newblock Aug 1999.
\newblock arXiv:math/9908099.

\end{thebibliography}
\end{document}